\date{\today}

\documentclass{compositio}
\usepackage{amsmath}
\usepackage{amssymb} 
\usepackage{mathrsfs}
\usepackage{todonotes}
\usepackage{xypic}
\usepackage{bigints}
\usepackage{mathdots}
\newtheorem{thm}{Theorem}
\newtheorem{lem}{Lemma}
\newtheorem{definition}{Definition}
\newtheorem{cor}[thm]{Corollary}
\newtheorem{prop}[thm]{Proposition}
\newtheorem{remark}{Remark}
\usepackage{color}
\classification{22E50, 11F66, 11F70}

\DeclareMathOperator{\Hom}{\textup{Hom}}

\DeclareMathOperator{\tr}{\textup{tr}_{\mathfrak{g}}}

\begin{document}
\title{On a Generalization of Local Coefficients }
\author{Carlos De la Mora}
\email{carlosdelamorac@gmail.com} %
\address{School of Mathematics\\University of East Anglia\\Norwich Research Park\\Norwich\\Norfolk\\ NR4 7TJ\\UK.}
\author{Shaun Stevens}
\email{shaun.stevens@uea.ac.uk} %
\address{School of Mathematics\\University of East Anglia\\Norwich Research Park\\Norwich\\Norfolk\\ NR4 7TJ\\UK.}

\maketitle

\section{Introduction}
We attempt to generalize the concept of local coefficients, local coefficients where first defined by Shahidi in \cite{Sh78} and have been studied extensively in consecutive works \cite{Shahidi81,Keys&Shahidi,ShPlancherel} to name a few. One of the reasons local coefficients are important is because they provide ways to define local $L$-factors and epsilon factors of the $L$--functions appearing in the Langlands-Shahidi method \cite{ShahidiBook,ShPlancherel}. Another important application is the relation that local coefficients have to the so called Plancherel measures (see Proposition \ref{Plancherel measures}). Local coefficients have only being defined in the case where we have a group that is quasi-split and for representations that are generic. In this paper we define what we call ``generalized local coefficients'', for the non-split group $\text{GL}_m(D)$, for $D$ a central division algebra over a non-archimedean local field $F$ of characteristic zero. The name ``generalized local coefficients'' is justified since in the case where $D=F$, we obtain that our generalized local coefficients are local coefficients defined by Shahidi. In fact, we developed a more general theory based on two hypothesis that we labelled $\mathbf{H}_1$ and $\mathbf{H}_2$ (Section \ref{Gen.local.coeff.}). Under these two hypothesis we defined generalized local coefficients and showed that in the quasi-split case they are a positive constant multiple of Shahidi's Local coefficients. We then showed that $\text{GL}_m(D)$ satisfies hypothesis $\mathbf{H}_1$ and $\mathbf{H}_2$.

Let us explain in more detail the ideas developed in this paper. Let $G$ be the $F$ points of a connected reductive group over $F$. We fix a maximal split torus $A$ in $G$. We let $Q=LU$ be a minimal parabolic subgroup defined over $F$, where the centralizer of $A$ in $G$ is $L$. Let $\mathfrak{g}$ denote the Lie algebra of $G$. We define $\Phi$ to be the set of roots coming form the adjoint action of $A$ on $\mathfrak{g}$. We let $\mathfrak{g}_{\alpha}$ be the Lie algebra that corresponds to a root $\alpha $. We let $\Delta$ denote the simple roots corresponding to $Q$. We say that a nilpotent element $Y\in \mathfrak{g}$, is \emph{relatively $\Delta$-regular} or just relatively regular for short if $Y$ is of the form $$Y=\sum\limits_{\alpha \in \Delta}Y_\alpha,\,Y_\alpha \neq 0,\,Y_{\alpha}\in \mathfrak{g}_\alpha.$$        
There is a one to one correspondence between subsets of $\Delta$ and standard parabolic subgroups of $G$. Let $\theta \subset \Delta$ and let $P=MN$ be the subgroup that corresponds to $\theta$. We let ${_GW}$ be the relative Weyl group with respect $A$. We use the notation $\overline{P}=M\overline{N}$ to denote the opposite parabolic to $P$. Let $K$ be the maximal compact subgroup of Bruhat--Tits \cite{Bruhat&Tits72} that satisfies $PK=G$. We consider $\tilde{w}$ an element in ${_GW}$ to be such that $\tilde{w}(\theta)=\theta'\subset\Delta$. We let $P'=M'N'$ be the parabolic subgroup that corresponds to $\theta'$. We let $\mathfrak{m}$, $\mathfrak{n}$ and $\overline{\mathfrak{n}}$ denote the Lie algebras of $M$, $N$ and $\overline{N}$, respectively. For an element $Z\in \mathfrak{g}$, we denote by $Z_\mathfrak{m}$ the projection of $Z$ into the Lie algebra of $M$. Given a relatively regular nilpotent element $Y$, we let $\varphi$ to be a co-character of $A$ with the property that $$\varphi(c)Y\varphi^{-1}(c)=c^2Y, c\in \overline{F}^{\times}.$$ 
We say that an irreducible representation $(\pi,V)$ of $G$, is \emph{$(Y,\varphi)$--generic}, if its space of degenerate Whittaker models $V_{Y,\varphi}\neq 0$ (Definition \ref{(Y,varphi) generic}). We get $Y_\mathfrak{m}=\sum_{\alpha \in \theta}Y_{\alpha}$. We say that an irreducible representation $(\sigma,W)$ of $M$, is \emph{$(Y,\varphi)$--generic} if $(\sigma,W)$ is $(Y_{\mathfrak{m}},\varphi)$--generic. We have thanks to the work of Moeglin and Waldspurger \cite{MW}, and by the work of Varma \cite{Varma}, that for a $(Y,\varphi)$--generic representation $(\pi,V)$ there exist for sufficiently large integers $n$, a sequence of compact open subgroups $G_n$ and characters $\chi_n$ of $G_n$, such that the $\chi_n$ isotypic component of $(\pi,V)$ when restricted to $G_n$ has fixed dimension equal to the dimension of degenerate Whittaker models $V_{Y,\varphi}$. Similarly for a $(Y,\varphi)$--generic representation $(\sigma,W)$ of $M$, we have for $n$ sufficiently large, a sequence of compact open subgroups $M_n$ and characters $\chi_n^M$ of $M_n$, such that the $\chi_n^M$ isotypic component of $(\sigma,W)$ when restricted to $M_n$ has fixed dimension equal to the dimension of degenerate Whittaker models $W_{Y_\mathfrak{m},\varphi}$. Let us denote by $V_n$ (resp. $W_n$) the underlying vector space of the $\chi_n$ (resp. $\chi_n^M$) isotypic component of $(\pi,V)$ (resp. $W_n$ ) when restricted to $G_n$ (resp. $M_n$). We say that $w$ and $(Y,\varphi)$ are compatible if $(\text{Ad}(w)Y)_{\mathfrak{m}}=Y_{\mathfrak{m}}$. 

Let $(\sigma,W)$ be an irreducible $(Y,\varphi)$-generic representation of $M$. Suppose that $(Y,\varphi)$ is compatible with $w$. Using the compatibility condition we get that the representation $({^w\sigma},W)$ of $M'$ is $(Y,\varphi)$--generic. We denote by $X^*(M)$ the lattice of algebraic characters of $M$ defined over $F$. Given an element $\nu \in X^*(M)\otimes_{\mathbb{Z}} \mathbb{C}$, we obtain an unramied character of $(\sigma,W)$ that we also denote by $\nu$ (Subsection \ref{Intertwining operators}). We denote by $(I(\nu,\sigma),V(\nu,\sigma))$ the representation of $G$ obtained by normalized parabolic induction form the representation $(\sigma\otimes\nu,W)$ of $M$. We also have a representation $I({^w\nu},{^w\sigma})$ obtained by normalized parabolic induction of the representation $({^w\sigma}\otimes {^w\nu},W)$. We have an intertwining operator $$A(\nu,\sigma,w):I(\nu,\sigma)\longrightarrow I({^w\nu},{^w\sigma}).$$ The hypothesis $\mathbf{H}_1$ will guarantee that $G_n \cap M=M_n$ and that the restriction of $\chi_n$ to $M$ is equal to $\chi_n^M$. Given a $v\in W_n$ we obtained by Proposition \ref{functions} a canonical function $f_{(\nu,\sigma,v)}\in V(\nu,\sigma)_n$, such that the restriction to $K$ is independent of $\nu$. Using the compatibility condition of $w$ and $(Y,\varphi)$ we are able to justify that the underlying vector space of the $\chi_n^M$ isotypic component of $(\sigma,W)$ when restricted to $M_n$ is the same as the underlying vector space of the $\chi_n^{M'}$ isotypic component of $({^w\sigma},W)$ when restricted to $M'_n$. We therefore obtain again form Proposition \ref{functions} the canonical function $f_{({^w\nu},{^w\sigma},v)}\in V({^w\nu},{^w\sigma})$. Making use of hypothesis $\mathbf{H}_2$ we obtain that every vector in $I(\nu,\sigma)_n$ is a canonical function of the form $f_{(\nu,\sigma,v)}$ for some $v\in W_n$. We have that $$A(\nu,\sigma,w)f_{(\nu,\sigma,v)}$$ must be in $I({^w\nu},{^w\sigma})_n$ and therefore has to be a canonical function. One of our main results is Theorem \ref{scalar}, where we proved that $f_{(\nu,\sigma,v)}\in V(\nu,\sigma)$ is an ``eigenvector'' for $A(\nu,\sigma,w)$ and the corresponding ``eigenvalue'' is what we call a generalized local coefficient. To be more precise, for sufficiently large $n$ we prove the existence of meromorphic function $D_{(Y,\varphi,n)}(\nu,\sigma,w)$ such that for $v\in W_n$ $$A(\nu,\sigma,w)f_{(\nu,\sigma,v)}= D_{(Y,\varphi,n)}(\nu,\sigma,w)^{-1}f_{({^w\nu},{^w\sigma},v)}.$$  The idea of finding such functions can be seen in the work of Keys \cite{Keys}, for principal series representations of $\text{SL}_2(F)$ and $\text{SU}_3(F)$. 

Two of the main properties of local coefficients is their relation to Plancherel measures and their multiplicativity property, Proposition \ref{Plancherel measures} and Corollary \ref{multiplicativity of local}, respectively. We have shown in Corollary \ref{gral. Plancherel measure} that the relation of generalized local coefficients to the Plancherel measures is a relation completely analogous to the one that local coefficents have to Plancherel measures. We also have shown in Corollary \ref{multiplicativity of geral.local} that generalized local coefficients are multiplicative.       
In Section \ref{relation to local coeff.}, we showed that in the quasi-split case, generalized local coefficients are a positive multiple of local coefficients and in the case of $\text{GL}_n(F)$ generalized local coefficients are local coefficients. In the last section we proved the existence of generalized local coefficients for the non quasi-split group $\text{GL}_m(D).$ 
\acknowledgements{The authors would like to state their gratitude to Shahidi, Waldspurger and Varma, for further explaining their work to us. This research was supported by EPSRC grant EP/H00534X/1.}
\section{Preliminaries and notation} \label{notation}

Let $F$ be a non-archimedean local field of characteristic zero, we denote by $\mathfrak{o}$ the ring of integers and by $\varpi$ a uniformizer in $F$. Let $\mathbb{G}$ be a connected reductive $F$-group and let $G=\mathbb{G}(F)$, we denote by $\overline{F}$ the algebraic closure of $F$ and we may identify $\mathbb{G}(\overline{F})$ with $\mathbb{G}$. Let $\mathfrak{g}$ be the lie algebra of $G$ and let $\tr$ be an invariant symmetric form on $\mathfrak{g}$. We have that $G$ acts on $\mathfrak{g}$ by an adiont action that we denote by $\text{Ad}$. Given $g\in G$, and $Z\in \mathfrak{g}$, we write $^gZ$ or $gZg^{-1}$ to mean $\text{Ad}(g)Z$. Similarly we write $Z^g$ or $g^{-1}Zg$ to mean $\text{Ad}(g^{-1})Z$. Let $A$ be a maximal split torus in $G$ and denote by $\Phi$ be the root system obtained by the action of $A$ on $\mathfrak{g}$. Let $X^*(A)$ and $X_{\ast}(A)$ denote characters and the co-characters of $A$ and denote by $\langle\,, \rangle$ the perfect paring between them. We let $Z_G(A)$ and $N_G(A)$ denote the centralizer and the normalizer of $A$ in $G$ respectively. We let $_GW$ denote the relative Weyl group $N_G(A)/Z_G(A)$. Let $\mathfrak{g}_{\alpha}=\lbrace X \in \mathfrak{g}: \text{Ad}(a)X=\alpha(a)X, a\in A \rbrace$. We then get a decomposition
$$\mathfrak{g}=\mathfrak{z}\bigoplus\limits_{\alpha \in \Delta}\mathfrak{g}_\alpha $$
Where $\mathfrak{z}$ denotes the lie algebra of $Z_G(A)$. Let $Y$ be a nilpotent element in $\mathfrak{g}$. Consider a co--character $\varphi:\mathbb{G}_m\rightarrow \mathbb{G}$ defined over $F$ such that $\text{Ad}(\varphi(c))Y=c^2Y$, for $c \in \overline{F}$. We would like to note that in \cite{Varma} and \cite{MW} they consider $\text{Ad}(\varphi(c))Y=c^{-2}Y$, but this change will not affect us in any serious way and it will help us to ease some notations. Given a parabolic subgroup $P=MN$ we denote by $\overline{P}=\overline{N}M$ the opposite parabolic subgroup to $P$. The co-character $\varphi$ induces a grading for $\mathfrak{g}$ $$\mathfrak{g}=\bigoplus\limits_{i \in \mathbb{Z}}\mathfrak{g}_i \text{ where } \mathfrak{g}_i=\lbrace X \in \mathfrak{g}: \text{Ad}(\varphi(c))X=c^{i}X, c\in F \rbrace $$
 Let 
 $$\mathfrak{q}=\bigoplus\limits_{i \geq 0}\mathfrak{g}_i,\, \mathfrak{u}=\bigoplus\limits_{i > 0}\mathfrak{g}_i,\,\overline{\mathfrak{u}}=\bigoplus\limits_{i < 0}\mathfrak{g}_i $$
Then $\mathfrak{q}$ is the lie algebra of some parabolic subgroup $Q(\varphi)=LU$, and $\mathfrak{u}$ (resp. $\overline{\mathfrak{u}}$) is the lie algebra of $U$ (resp. $\overline{U}$)--the unipotent radical of $Q(\varphi)$(resp. the unipotent radical of $\overline{Q}$).
  
Let $\tr$ denote a non-degenerate symmetric bilinear form on $\mathfrak{g}$ that is invariant under the adjoint action of $G$. In an attempt to be consistent with the notations in \cite{MW,Varma} we write $\tr(XZ)$ to denote $\tr(X,Z)$. Let $Y^{\#}=\lbrace X\in \mathfrak{g}: [Y,X]=0\rbrace=\lbrace X\in \mathfrak{g}: \tr(Y[X,Z])=0\text{ for all } Z\in \mathfrak{g}\rbrace$. Let $B_Y$ denote the alternating bilinear form on $\mathfrak{g}\times\mathfrak{g}$ given by $B_Y(X,Z)=\tr(Y[X,Z])$. Then $B_Y$ induces a non-degenerate alternating bilinear form on $\mathfrak{g}/Y^{\#}\times \mathfrak{g}/Y^{\#}$ that we still denote by $B_Y$. 
\medskip

\begin{lem} \label{lattice}There is a lattice $\mathcal{L}$ of $\mathfrak{g}$ that satisfies:  
\begin{itemize}

\item[1)] $\mathcal{L}=\oplus_{i\in \mathbb{Z}}\mathcal{L}_i \text{ where } \mathcal{L}_{i}=\mathfrak{g}_{i}\cap \mathcal{L} $,
\item[2)] $ \mathcal{L}_Y=\mathcal{L}/(\mathcal{L}\cap Y^{\#}) $ is self dual with respect to $\psi\circ B_Y$. Where $\psi$ is a character of $F$ of level zero (trivial in $\mathfrak{o}$ but not in $\varpi^{-1}\mathfrak{o}$). In other words $ \mathcal{L}_Y=\lbrace x\in \mathfrak{g}/Y^{\#}| \psi\circ B_Y(x,z)=1,\text{ for all }z\in \mathcal{L}_Y  \rbrace$.  
\end{itemize}     

\end{lem}

Let $f$ be a function on $\mathfrak{g}$ that is locally constant of compact support. We denote by $\mu_{\mathfrak{g}}$ the Haar measure on $\mathfrak{g}$. We then define $$\widehat{f}(Y)=\int\limits_{\mathfrak{g}}f(X)\psi(\tr(X,Y))d\mu_{\mathfrak{g}}(Y),\, Y\in \mathfrak{g}.$$   
We get that $\widehat{f}$ is a function on $\mathfrak{g}$ that is locally constant of compact support. Let $\mathcal{O}$ denote the orbit of $X \in \mathfrak{g}$ under the adjoint action. We let $Stab_G(X)$ denote the Stabilizer under this action. It is known that $Stab_G(X)$ is unimodular and therefore induces an invariant measure on $G/Stab_G(X)$ unique up to a constant. We have by a result Rao, also attributed independently to Deligne \cite{Rao,DeBacker&Sally}, that the integral $$\mu_{\mathcal{O}}(f)=\int_{G/Stab_{G}(X)}f(\text{Ad}(x)X)dx$$ converges for $f$ locally constant of compact support on $\mathfrak{g}$. We then define $\widehat{\mu_{\mathcal{O}}}(f)=\mu_{\mathcal{O}}(\widehat{f})$. Let $(\pi,W)$ be an irreducible smooth representation of $G$. There is a theorem of Harish--Chandra \cite{DeBacker&Sally}, that shows that 
\begin{equation}
\text{tr}(\pi)=\sum\limits_{\mathcal{O}\in \mathcal{O}(\mathfrak{g})} C_{\mathcal{O}}\widehat{\mu_{\mathcal{O}}} \label{Harish_Chandra character formula} \tag{$\ast$}
\end{equation}
Where $\mathcal{O}(\mathfrak{g})$ denotes the nilpotent orbits of $\mathfrak{g}$ and $C_{\mathcal{O}}$ are complex numbers. We can be more explicit about the Harish--Chandra character formula (\ref{Harish_Chandra character formula}). Indeed, there exists an open set $\mathscr{U}$ around $0\in \mathfrak{g}$, homeomorphic to an open set $\mathscr{V}$ of the identity in $G$ under the exponenital map $\exp:\mathscr{U}\longrightarrow \mathscr{V}$, the inverse is given by the map $\log:\mathscr{V}\longrightarrow \mathscr{U} $. Then, for every continuous function $f$ of compact support in $\mathscr{V}$ we get 

										$$\text{tr}(\pi(f))=\sum\limits_{\mathcal{O}\in \mathcal{O}(\mathfrak{g})} C_{\mathcal{O}}\widehat{\mu_{\mathcal{O}}}(f\circ\text{exp})$$

We impose an order on the set of nilpotent orbits by saying $\mathcal{O}'\leq \mathcal{O}$, if $\mathcal{O}'\subset \overline{\mathcal{O}}$. Where $\overline{\mathcal{O}}$ denotes the closure in the topology coming from $\mathfrak{g}$.
\begin{definition}\label{Y--generic} Let $(\pi,W)$ be an irreducible representation of $G$, and let $\mathcal{O}_Y$ denote the orbit of $Y$. We say that that $(\pi,W)$ is $Y$--\emph{generic}, if $\mathcal{O}_Y$ is maximal with respect to the property that the coefficient $c_{\mathcal{O}_Y}$ in the Harish-Chandra character expansion of $(\pi,W)$ like above is not zero.
\end{definition}
The fact that Definition \ref{Y--generic} is indeed a generalization of generic representations for quasi-split groups will be clear from Theorem \ref{2}. The next theorem summarizes some of the results in \cite{MW,Varma}.        

\begin{thm}\label{G_n} Let $Y$, $Q(\varphi)$ and $\mathcal{L}$ be as before. There is a positive integer $B$ such that for all integers $n>B$.  

\begin{itemize}
\item[1\text{)}] $\exp(\varpi^n\mathcal{L})$ is a subgroup of $G$ that we denote by $G_n$,
\item[2\text{)}] $G_n$ has an Iwahori decomposition with respect to every standard parabolic.  
\end{itemize}
Let $(\pi,W)$ be a $Y$--generic representation of $G$ and suppose further that $\mathcal{O}_Y$ is maximal. After choosing an appropriate measure for $G$, which induces a measure on $\mathfrak{g}$ as in \cite[1039]{Varma}, we get that there exists a character $\chi_n$ of $G_n$, trivial on $Q(\varphi)\cap G_n$, such that the vector space $W_n=\lbrace w\in W:\pi(x)w=\chi(x)w \rbrace$ has dimension $C_{\mathcal{O}_Y}$.
\end{thm}

The last paragraph of Theorem \ref{G_n} corresponds to Lemma 7 in \cite{Varma}. The groups $G_n$ and the characters $\chi_n$ of Theorem \ref{G_n} will be crucial in the rest of the paper. We will be more explicit about how the characters $\chi_n$ are defined, from now we just introduce some more notation, we write $X_n$ to denote $X\cap G_n$, for any subset $X$ of $G$. If $(\pi,W)$ is a representation of $G$, we write $(\pi,W)^{G_n,\chi_n}$ or just by $W_n$ if there is no confusion, to denote the $\chi_n$ isotypic component of the representation $(\pi,W)$ restricted to $G_n$.

We fix from now on a minimal parabolic subgroup $Q=LU$, defined over $F$, with maximal split component $A$. Let $\Delta$ denote the set of simple roots in $X_{\ast}(A)$ corresponding to $Q$. We also have a set of positive roots that we denote by $\Phi^{+}$, and a set of negative roots that we denote by $\Phi^{-}$, determined by $Q$. 
\begin{definition} Let $Y\in \mathfrak{g}$ be nilpotent element. We say that $Y$ is a \emph{relatively regular (relatively $\Delta$--regular)} nilpotent element, if is of the form $$\sum\limits_{\alpha \in \Delta}Y_\alpha\text{ where }Y_\alpha \in \mathfrak{g_\alpha}\,\, Y_\alpha\neq 0.$$ Since we work with a fixed $Q$ and thus a fixed $\Delta$, we usually drop the symbol $\Delta$ and just talk about \emph{relatively regular} nilpotent element.\end{definition} We have that there is a co-character $\varphi$ in $X_{\ast}(A)$, such that $\langle \alpha, \varphi \rangle =2$, for all $\alpha \in \Delta$. Therefore the character $\varphi$, has to satisfy that for $\alpha \in \Phi^{+}$, $\langle \alpha, \varphi \rangle >0$ and for $\alpha \in \Phi^{-}$, $\langle \alpha, \varphi \rangle <0$. We conclude that $Q(\varphi)=Q$. If for a relatively regular nilpotent element $Y$, we have a $\varphi \in X_{\ast}(A)$, such that $\text{Ad}(\varphi(c))Y=c^{2}Y$, for $c\in \overline{F}$, we get $\langle \alpha, \varphi \rangle =2$, for all $\alpha \in \Delta$, and thus $Q(\varphi)=Q$. We denote by $(Y,\varphi)$ a pair given by a relatively regular nilpotent element $Y$, and a co-character $\varphi \in X_{\ast}(A)$, such that $\text{Ad}(\varphi(c))Y=c^{2}Y$, for $c\in \overline{F}$.

We obtain a character $\chi:\overline{U}\longrightarrow \mathbb{C}^{\times}$ given by $\chi(\gamma)=\psi(\tr(Y\log(\gamma)))$. We denote by $W_{Y,\varphi}$ (or $W_{Y,\,\overline{U}}$) the twisted Jaquet functor obtained by taking the quotient of $W$ by the space spanned by $$\lbrace\pi(\bar{u})w-\chi(\bar{u})w\rbrace,\,\bar{u}\in \overline{U}.$$ Using the definition given \cite[Pg. 428]{MW} or the definition given in the introduction of \cite[Pg. 1028]{Varma} in the case where $\mathfrak{g}_1=0$, we define the vector space of \emph{degenerate Whittaker forms} to be the space $W_{Y,\,\overline{U}}$. We note that in our situation $\mathfrak{g}_1=0$ because $Y$ is relatively regular nilpotent. 

There is a connection that we need to specify between $\chi$ and the characters $\chi_n$ in Theorem \ref{G_n}. Let $(Y,\varphi )$ be the pair used to define $\chi$, and let $\varphi(\varpi)=t$. Then $\chi_n(qj)=\chi(t^njt^{-n})$ for $q\in Q_n,\, j\in \overline{U}_n$. We get
\begin{align*}\chi_n(qj)=&\chi(t^njt^{-n})=\psi(\tr(Y\log(t^njt^{-n})))\\
=&\psi(\tr(t^{-n}Yt^{n}\log(j)))=\psi(\tr(\varpi^{-2n}Y\log(j)))
\end{align*}
We also point out that $\chi_n$ is trivial in $Q_n$.
\begin{definition} \label{(Y,varphi) generic} Let $(\pi, W)$ be a representation of $G$, and let $\mathcal{O}_Y$ denote the orbit of $Y$. We say that that $(\pi,W)$ is  $(Y,\varphi)$--\emph{generic}, if $\mathcal{O}_Y$ is maximal with respect to the property that $W_{Y,\varphi}\neq 0$. 
\end{definition}

In the case where $Y$ is relatively regular we have that $\mathcal{O}_Y$, the orbit of $Y$ under the adjoint action of $G$, is a maximal orbit. We then get out of Theorem 1 in \cite{Varma} the following result.  
\begin{thm}\label{2} 

\begin{itemize}
\item[]
\item[i)]The representation $(\pi,W)$ of $G$ is $Y$--generic for a relatively regular nilpotent $Y$ if and only if is $(Y,\varphi)$--generic some $\varphi$.
\item[ii)]Let $(\pi,W)$ of $G$ be $(Y,\varphi)$--generic and let $B$ and $W_n$ be as in Theorem \ref{G_n}. The dimension of the space of degenerate Whittaker forms is equal to $C_{\mathcal{O}_Y}$ and thus for $n>B$,  equal to the dimension of the space $W_n$.
\end{itemize}
 
\end{thm}
We then introduce the convention that every time we use the objects $W_n$, $G_n$ or $\chi_n$, we mean that $n$ is sufficiently large so that Theorem \ref{G_n} (and thus Theorem \ref{2}) is  satisfied.
\begin{remark} We do not need the element $Y$ to be relatively regular to have a good notion of degenerate Whittaker forms with respect to some pair $(Y,\varphi)$. We could have changed Definition \ref{2} to say that a representation $(\pi,W)$ is  $(Y,\varphi)$--\emph{generic}, if $\mathcal{O}_Y$ is maximal with respect to the property that the space of degenerate Whittaker models with respect to some pair $(Y,\varphi)$ is non--zero. With this new definition it will follow from the work of \cite{MW,Varma} that the notion of $Y$--generic and $(Y,\varphi)$--generic are equivalent. We did not go through with this more general situation because the definition of degenerate Whittaker forms for an arbitrary nilpotent element $Y$ is considerably more complicated and will not help us in the sequel. 
\end{remark}
\subsection{Intertwining Operators}\label{Intertwining operators}

Let $P=MN$ be a parabolic subgroup with left Haar measure $\mu_P$. Define $\delta_P$ to be the character that satisfies the formula $$\delta_P(g)\int f(xg)d\mu_P(x)=\int f(x)d\mu_P(x),\,g\in P.$$ We get that $\delta_P(m)=\text{det}|\text{Ad}_{\mathfrak{n}}(m^{-1})|$. Let $X^*(M)$ be the algebraic characters of $M$ defined over $F$. We have a map  
										$$H_M:M\longrightarrow \mathfrak{a}=\Hom(X^*(M),\mathbb{R})$$ 
defined by 
										$$ q^{\langle \chi, H_M(m)\rangle}=|\chi(m)|,$$
$\chi\in X^*(M),\,m\in M$. We let $\mathfrak{a}^{*}=X^{*}(M)\otimes_{\mathbb{Z}}\mathbb{R}$, which is the dual of $\mathfrak{a}$. We let $\mathfrak{a}_{\mathbb{C}}^*=\mathfrak{a}\otimes_{\mathbb{R}}\mathbb{C}$. Given $\nu \in \mathfrak{a}_{\mathbb{C}}^*$, we define an unramified character of $M$, that we also denote by $\nu$, by the formula $\nu(m)=q^{\langle \nu, H_M(m)\rangle}$.

We get for any subset $\theta$ of $\Delta$, we get in the usual way a standard parabolic subgroup $P_\theta=M_\theta N_\theta $. Let $\theta'$ be also a subset of $\Delta$ such that there is an element $\tilde{w}\in {_GW}$ satisfying $\tilde{w}(\theta)=\theta'$. We also have a parabolic subgroup $P_{\theta'}=M_{\theta'} N_{\theta'} $ associated to $\theta'$ and thus satisfies that $^wM=M'$. Let $w\in G$, be a representative of $\tilde{w}$ in $_GW$. Let $(\sigma ,W)$ be an irreducible representation of $M$ and choose $\nu \in \mathfrak{a}_{\mathbb{C}}^* $, we define $$I(\nu,\sigma)=\text{Ind}_{P_\theta}^G(\sigma \otimes \nu\delta_{P_{\theta'}}^{-1/2})$$ 
We denote by $V(\nu,\sigma)$ the space of functions where $I(\nu,\sigma)$ acts by right translations. Consider the representation $^w\sigma$ and the unramified character $^w\nu$ of $M'$ given by $$^w\sigma(x)=\sigma(w^{-1}xw),\text{ and } ^w\nu(x)=\nu(w^{-1}xw).$$ We then have an induced representation $$I(^w\nu,{^w\sigma})=\text{Ind}_{P_{\theta'}}^G(w^\sigma \otimes ^w\nu\delta_{P_{\theta'}}^{-1/2}).$$ 
We define the {\it intertwining operator} $A(\nu,\sigma,w)$ first by $$A(\nu,\sigma,w)f(g)=\int\limits_{w\overline{N}w^{-1}\cap U}f(w^{-1}ng)dn,\text{ for }f\in V(\nu,\sigma),$$ 
where we know that the integral is guaranteed to converge absolutely on a positive cone in $\mathfrak{a}_{\mathbb{C}}^*$. If for a fixed $\nu$ the integral defining $A(\nu,\sigma,w)$ converges for every $f\in V(\nu,\sigma)$ and every $g\in G$, we get that $A(\nu,\sigma,w)$ is an intertwining operator from $I(\nu,\sigma)$ into $I(^w\nu,{^w\sigma})$. Moreover, letting $\nu$ vary we get that $A(\nu,\sigma,w)$ is meromorphic on the variable $\nu$ and thus can be extended by analytic continuation to some bigger open set. Let us be more precise on the use of the word meromorphic. We follow closely the explanation of term meromorphic given in \cite[IV.1]{Waldspurger}. The set $\mathcal{X}(M)$ of unramified characters of $M$ is a complex algebraic variety. Let $\mathfrak{B}$ denote the space of polynomial functions of $\mathcal{X}(M)$. Let 			      $$\mathcal{O}_\mathbb{C}=\left\lbrace \sigma\otimes\nu:\nu \in \mathfrak{a}_{\mathbb{C}}^* \right\rbrace.$$ 
The variety $\mathcal{X}(M)$ induces a structure of complex algebraic variety on $\mathcal{O}_\mathbb{C}$. To ease the notation we write $P,P'$ instead of $P_{\theta},P_{\theta'}$, respectively. We have the existence of a maximal compact subgroup $K$ of $G$, with the property that $G=PK=P'K$. We then have an isomorphism of vector spaces from $I(\nu,\sigma)$ (resp. $I(^w\nu,{^w\sigma})$) into $\text{Ind}_{K\cap P}^K \sigma$ (resp. $\text{Ind}_{K\cap P}^K \sigma$) given by restriction $\text{Res}^P_K$ (resp. $\text{Res}^{P'}_K)$. Saying that $A(\nu,\sigma,w)$ is meromorphic means that there exists a Zarinski open set $\mathcal{U}$ of $\mathcal{O}_{\mathbb{C}}$, an element $b\in \mathfrak{B}$, such that for all $\sigma\otimes \nu \in \mathcal{U}$ and all $f\in \text{Ind}_{K\cap P}^K$, there exist $\left\lbrace f_1,\ldots f_r\right\rbrace$ and $\left\lbrace b_1,\ldots b_r\right\rbrace$ satisfying $$b(\nu)\text{Res}^{P'}_K\circ A(\nu,\sigma,w)\circ(\text{Res}^P_K)^{-1}f=\sum_{i=1}^rb_i(\nu)f_i$$  

\subsubsection{Restriction of Intertwining Operators.} Let $\theta_{\ast}\subset \theta$, then $\tilde{w}(\theta_{\ast})\subset\Delta$. We let $P_{\ast}=M_{\ast}N_{\ast}$ and $P'_{\ast}=M'_{\ast}N'_{\ast}$, be the standard parabolic subgroups corresponding to $\theta_{\ast}$ and $\tilde{w}(\theta_{\ast})$ respectively. We note that $P_{\ast}\subset P$, $M_\ast \subset M$, and $N\subset N_{\ast}$. Let $(\sigma_{\ast},W_\ast)$ be a representation of $M_{\ast}$. Let $\mathfrak{a}^*_{\ast}=X(A_{\ast})\otimes_{\mathbb{Z}} \mathbb{R}$, and let $\nu_{\ast}\in \mathfrak{a}^*_{\ast,\mathbb{C}}=\nu_{\ast}\in \mathfrak{a}^*_{\ast}\otimes_{\mathbb{R}} \mathbb{C}$. Let us denote by $ I_M(\nu_{\ast},\sigma_{\ast})$ the representation obtained by normalized induction $$\text{Ind}_{P_{\ast}\cap M}^M(\sigma_{\ast} \otimes \nu_{\ast}\delta_{P_{\ast}\cap M}^{-1/2}).$$ For $\nu\in \mathfrak{a}^*_{\mathbb{C}} $, we set $\tilde{\nu} \in \mathfrak{a}^*_{\ast,\mathbb{C}} $ to satisfy $$\left\langle \tilde{\nu}, H_{M_\ast}(m_{\ast}) \right\rangle= \left\langle \nu, H_{M}(m_{\ast}) \right\rangle,\,\forall \, m_{\ast}\in M_{\ast}. $$
Suppose there is an injection $T:(\sigma,W)\longrightarrow I_M(\nu_{\ast}, \sigma_{\ast})$, then $T$ 
induces an  injection $$T_{\ast}(\nu):I(\nu,\sigma)\longrightarrow I(\nu_{\ast}+\tilde{\nu},\sigma_{\ast}).$$ We have that $T_{\ast}(\nu)$ is given explicitly by the following $(T_{\ast}(\nu)f)(g)=T(f(g))(1), \,f\in V(\nu,\sigma),\,g\in G$. We also have that $T$ induces an injection $T^w:({^w\sigma},W)\longrightarrow I_{M'}({^w\nu_{\ast}},{^w\sigma_{\ast}}),$ where $I_{M'}({^w\nu_{\ast}},{^w\sigma_{\ast}})$ comes from normalized induction $$\text{Ind}_{P'_{\ast}\cap M'}^{M'}({^w\sigma_{\ast}} \otimes {^w\nu}_{\ast}\delta_{P'_{\ast}\cap M'}^{-1/2}).$$   
Explicitly $T^w$ is given by $T^w(v)(m')=T(v)(wm'w^{-1}),\, v\in W,\, m\in M'.$ Then similarly to $T$, we get that $T^w$ induces an injection $$T^w_{\ast}(\nu):I({^w\nu},{^w\sigma})\longrightarrow I({^w\nu_{\ast}}+{^w\tilde{\nu}},{^w\sigma_{\ast}}). $$ The map $T^w_{\ast}(\nu)$ is given by $(T^w_{\ast}(\nu)f)(g)=T^w(f(g))(1)=T(f(g))(1), \,f\in V({^w\nu},{^w\sigma}),\,g\in G.$ 

\begin{lem}\label{commutativity} Using the notation above we have that the following diagram$$\xymatrixcolsep{7pc}\xymatrix{
       I({\nu_{\ast}}+{\tilde{\nu}},{\sigma_{\ast}})\ar[r]^{A(\tilde{\nu}+\nu_{\ast},\sigma_{\ast},w)}  &I({^w\nu_{\ast}}+{^w\tilde{\nu}},{^w\sigma_{\ast}}) \\
        I({\nu},{\sigma})\ar[r]_{A(\nu,\sigma,w)}\ar[u]^{T_{\ast}(\nu)}       & I({^w\nu},{^w\sigma})\ar[u]^{T^w_{\ast}(\nu)} }$$ is commutative.
\end{lem}

\begin{proof} Let $f\in V(\nu,\sigma)$. We need to show that $$A(\tilde{\nu}+\nu_{\ast},\sigma_{\ast},w)(T_{\ast}(\nu)(f))(g)=T^w_{\ast}(\nu)(A(\nu,\sigma,w)f)(g),\,\forall g\in G.$$
Suppose for a moment that $w\overline{N_{\ast}}w^{-1}\cap U=w\overline{N}w^{-1}\cap U$. We then have
\begin{align*}&A(\tilde{\nu}+\nu_{\ast},\sigma_{\ast},w)(T_{\ast}(\nu)(f))(g)=\int\limits_{w\overline{N_{\ast}}w^{-1}\cap U}(T_{\ast}(\nu)(f))(w^{-1}ng)dn=\int\limits_{w\overline{N_{\ast}}w^{-1}\cap U}T(f(w^{-1}ng))(1)dn\\
&=T\left(\,\,\int\limits_{w\overline{N}w^{-1}\cap U}f(w^{-1}ng) dn\right)(1)=T((A(\nu,\sigma,w)f)(g))(1)=T_{\ast}^w(\nu)(A(\nu,\sigma,w)f)(g)
\end{align*}
It remains to show that $w\overline{N_{\ast}}w^{-1}\cap U=w\overline{N}w^{-1}\cap U$. The containment  $\overline{N_{\ast}}\supset\overline{N}$, implies the containment $w\overline{N_{\ast}}w^{-1}\cap U\supset w\overline{N}w^{-1}\cap U$. For the reverse containment we have $\overline{N_{\ast}}=(M\cap\overline{N_{\ast}}) \overline{N}$. Then   $$ w\overline{N_{\ast}}w^{-1}\cap U=\left(w(M\cap \overline{N_{\ast}})w^{-1}\cdot(w\overline{N}w^{-1})\right)\cap U.$$ Let $x \in w(M\cap \overline{N_{\ast}})w^{-1},\,y\in w\overline{N}w^{-1},$ be such that $xy\in U$. The condition $\tilde{w}(\theta)\subset \Delta$ implies $x\in \overline{U}$. We deduce out of Proposition 21.9 in \cite{BorelBook} that we can write $y=y_1y_2$, $y_1\in w\overline{N}w^{-1}\cap \overline{U}$ and $y_2\in w\overline{N}w^{-1}\cap U$. Since $xy_1y_2\in U$ we get that the product $xy_1$ is in $U\cap \overline{U}$ and therefore equal to the identity. We get $x=y_1^{-1},$ which implies $x$ and $y_1$ are the identity element. Therefore $xy=y_2\in w\overline{N}w^{-1}\cap U$, hence the reverse containment and the proof of the lemma.

\end{proof}

\begin{remark} A very similar statement to the one of the Lemma \ref{commutativity} can be found on \cite[Pg. 87]{ShahidiBook} and on \cite[Pg. 329]{Shahidi81}. We decided to include our own poof and version of Lemma \ref{commutativity} instead of citing it, because we believe it provides clarity for the proof of Theorem \ref{scalar}. 
\end{remark}

\subsubsection{Duality between $V(\nu,\sigma)$ and $V(-\overline{\nu},\sigma)$.}\label{duality}
In the case that $(\sigma,W)$ is unitary it is possible to define a duality between $I(\nu,\sigma)$ and $I(-\overline{\nu},\sigma)$. Indeed, take $(\,,)$ to be the Hermitian form in $W$. Let $K$ be a maximal compact subgroup with the property that $G=PK$. Then for $f_1\in V(\nu,\sigma) $ and $f_2\in V(-\overline{\nu},\sigma)$, we define $$\langle f_1\,,f_2 \rangle =\int\limits_K(f_1(k),\,f_2(k))dk .$$
Where we integrate with respect to the Haar measure on $K$. We see that $\langle\,, \rangle$ is equivalent to the duality defined by Shahidi in \cite[5.2]{ShahidiBook}. 

\subsection{Local Coefficients. }\label{local coefficients}

We continue with the same notation as in the subsection \ref{Intertwining operators}. As we mentioned before, local coefficients have only been defined in the case where $G$ is quasi-split. We thus, in this section, restrict to the case where $G$ is quasi-split. Let $\widetilde{w}_{\ell}\in {_GW}$ (resp. $\widetilde{w}^M_{\ell} \in {_MW}$) be the longest element in $_GW$ (resp. $_MW$, the relative Weyl group of $M$). Let $\widetilde{w}_{0}=\widetilde{w}_{\ell}\widetilde{w}^M_{\ell}$ and let $w_{0}$ be a representative in $G$ of $\widetilde{w}_0$. Let $\chi$ be a non--degenerate character of $U$, and let $\widetilde{w}\in {_GW}$ be such that $\widetilde{w}(\theta)\subset \Delta$. We say that $w$ and $\chi$ are compatible if $\chi(wuw^{-1})=\chi(u)$, for all $u\in M\cap U$. We say that $\sigma$ is $\chi${\it-generic} if $\Hom_{M\cap U}\left(\sigma, \chi\right)$ is not zero, in which case it is 1-dimensional. We call the elements of $\Hom_{M\cap U}\left(\sigma, \chi\right)$ {\it Whittaker functionals}. The fact that the space of Whittaker functionals is one dimensional is what gives rise to the local coefficients. Indeed, suppose that $w_0$ and $\chi$ are compatible. Given $\lambda$ a Whittaker functional we can construct a canonical functional $\lambda_{\chi}(\nu,\sigma)\in \Hom_{U}(I(\nu,\sigma),\chi)$ by the formula 

					$$\int\limits_{w_0\overline{N}w_0^{-1}}\lambda(f(w_0^{-1} n))\overline{\chi(n)}dn.$$
Suppose that $\chi$ and $w$ are compatible. Let $w_{\ell}^{M'}=ww_{\ell}^Mw^{-1}$, then $w_{\ell}^{M'}$ is a representative of the longest element in $_{M'}W$. We let $w_0'=w_{\ell}w_{\ell}^{M'}$ and suppose that $w_0'$ and $\chi$ are compatible. We then have a canonical Whittaker functional $\lambda_{\chi}({^w\nu},{^w\sigma})\in \Hom_{U}(I(^w\nu,{^w\sigma}), \chi)$, given by the formula 

$$\int\limits_{w_0'\overline{N'}{w_0'}^{-1}}\lambda(f({w_0'}^{-1} n'))\overline{\chi(n')}dn'.$$
We have that $\lambda_{\chi}(^w\nu,{^w\sigma})A(\nu,\sigma,w)\in \Hom_{U}(I(\sigma, \nu),\chi).$ Since any two Whittaker functionals are proportional we have that there exists $C_{\chi}(\nu,\sigma,w)\in \mathbb{C}\cup \{\infty\}$  such that 
$$\lambda_{\chi}({\nu},{\sigma})=C_{\chi}(\nu,\sigma,w) \lambda_{\chi}(^w\nu,{^w\sigma})A(\nu,\sigma,w) $$
The function $C_{\chi}(\nu,\sigma,w)$ is the local coefficient attached to $\nu,\sigma, \chi$ and $w$. The definition is due to Shahidi \cite{Shahidi81},\cite[5.1]{ShahidiBook}.

\subsection{Plancherel measures.}
We follow closely Shahidi's book \cite[5.3]{ShahidiBook} for this subsection. Let $\sigma$ be an irreducible unitary $\chi$-generic representation of $M$. Let $\nu \in \mathfrak{a}_{\mathbb{C}}^*$ and consider $$A(\nu,\sigma,w):I(\nu,\sigma)\longrightarrow I({^w\nu},{^w\sigma})$$
as well as $$A({^w\nu},{^w\sigma},w^{-1}):I({^w\nu},{^w\sigma})\longrightarrow I(\nu,\sigma).$$
Assume that $\nu$ is so that $I(\nu,\sigma)$ is irreducible. Then by Schur's Lemma, 
$$A:A({^w\nu},{^w\sigma},w^{-1})A(\nu,\sigma,w):I(\nu,\sigma)\longrightarrow I(\nu,\sigma)$$
is a scalar operator.
Let $$\gamma_w(G/P)=\int\limits_{\widetilde{N}_{w}}\delta_P(\tilde{n})^{-1}d\tilde{n}$$
where $$ \widetilde{N}_{w}=\overline{N}\cap w^{-1}Uw$$
We also define in an analogous manner $\gamma_{w^{-1}}(G/P')$. We then define a complex number $\mu(\nu,\sigma,w)$ to satisfy $$A({^w\nu},{^w\sigma},w^{-1})A(\nu,\sigma,w)=\mu(\nu,\sigma,w)\gamma_w(G/P)\gamma_{w^{-1}}(G/P')$$
In analogy to the tempered case and when $w=w_0$ this is what Shahidi calls the \emph{Plancherel measure} attached to $\nu,\sigma$, and $w$. We get from Corollary 5.3.1 in \cite{ShahidiBook} the following result. 

\begin{prop} \label{Plancherel measures} Let $(\sigma,W)$ be $\chi$-generic. Suppose that $\chi$ is compatible with $w$,$w_0$ and $w_0'$. One then has 
$$\mu(\nu,\sigma,w)\gamma_w(G/P)\gamma_{w^{-1}}(G/P')=C_{\chi}(\nu,\sigma,w)^{-1}C_{\chi}({^w\nu},{^w\sigma},w^{-1})^{-1}$$
\end{prop}    

\subsection{Factorization of intertwining operators.}
The results of this section are taken from \cite[4.2]{ShahidiBook}. Let $\theta,\theta' \subset \Delta .$ Let $$W(\theta,\theta')=\lbrace \tilde{w}\in {_GW}|\,\tilde{w}(\theta)=\theta'\rbrace .$$ We say that $\theta$ and $\theta'$ are \emph{associate} if $W(\theta,\theta')$ is not empty. Let $\alpha \in \Delta -\theta $, let $\Omega= \theta \cup \lbrace \alpha \rbrace $. Let $M_{\Omega}$ and $M_{\theta}$ be the Levi subgroups corresponding to $\Omega$ and $\theta$, respectively. Define $\overline{\theta}=\tilde{w}_{\ell,\Omega}\tilde{w}_{\ell,\theta}(\theta) \subset \Omega$, where $\tilde{w}_{\ell,\Omega}$ and $\tilde{w}_{\ell,\theta}$ are the longest elements of the Weyl groups of ${_{M_{\Omega}}W}$ and ${_{M_{\theta}}W}$, respectively. We call $\overline{\theta}$ the conjugate of $\theta$ in $\Omega$. The following theorem comes from putting together Lemma 4.2.1 and Theorem 4.2.2 in \cite{ShahidiBook}.

\begin{thm} \label{factorization of intertwining} Suppose that $\theta$ and $\theta'$ are associate. Take $\tilde{w} \in W(\theta,\theta').$ Then there exists a family of subsets $\theta_1,\theta_2, \ldots ,\theta_k \in \Delta $ such that 
\begin{itemize}
\item[a\text{)}]$\theta_1=\theta $ and $\theta_k=\theta'$;
\item[b\text{)}] fix $1 \leq i \leq k-1$; then there exists $\alpha_{i}\in \Delta -\theta_i$ such that $\theta_{i+1}$ is the conjugate of $\theta_i$ in $\Omega_i=\theta_i\cup \lbrace \alpha_{i}\rbrace $;
\item[c\text{)}] set $\tilde{w}_i=\tilde{w}_{\ell,\Omega_{i}}\tilde{w}_{\ell,\theta_{i}}$ in $W(\theta_i,\theta_{i+1})$ for $1\leq i < k,$ then $\tilde{w}=\tilde{w}_{k-1}\ldots \tilde{w}_1$;
\end{itemize}
Let $\nu\in \mathfrak{a}_{\theta,\mathbb{C}}^*$ be in the cone of absolute convergence of $A(\nu,\sigma,w)$. Then each $\nu_i\in \mathfrak{a}_{\theta_{i},\mathbb{C}}^*$ is in the cone of absolute convergence for $A(\nu_i,\sigma_i,w_i)$, where $\nu_1=\nu$, $\nu_i={^{w_{i-1}}\nu_{i-1}}$, $\sigma_1=\sigma$ and $\sigma_i={^{w_{i-1}}\sigma_{i-1}},$ for $2\leq i\leq k-1$. Moreover, $$A(\nu,\sigma,w)=A(\nu_{k-1},\sigma_{k-1},w_{k-1})\cdots A(\nu_1,\sigma_1,w_1).$$   

\end{thm}

\begin{cor}[(Multiplicativity of local coefficients)]\label{multiplicativity of local}One has $$C_{\chi}(\nu,\sigma,w)=\prod\limits_{i=1}^{n-1}C_{\chi}(\nu_i,\sigma_i,w_i).$$ 
\end{cor}

\section{Generalized local coefficients }\label{Gen.local.coeff.}

We let $P=MN$ be a parabolic subgroup of $G$ defined over $F$, and we let $\theta\subset \Delta$ the subset corresponding to $P$. We let $\mathfrak{m},\mathfrak{n}$ and $\overline{\mathfrak{n}}$ denote the lie algebra of $M,N$ and $\overline{N}$, respectively. We then have that $\mathfrak{g}=\overline{\mathfrak{n}}\oplus\mathfrak{m}\oplus\mathfrak{n}$. Given an element $Z\in \mathfrak{g}$, we denote by $Z_{\mathfrak{m}}$ the image of $Z$ under the projection of $\mathfrak{g}$ onto $\mathfrak{m}$. We let $(Y,\varphi)$ be as before, in particular $Y=\sum\limits_{\alpha \in \Delta}Y_\alpha,\,Y_{\alpha}\neq 0$. We then get $Y_{\mathfrak{m}}=\sum\limits_{\alpha \in \theta}Y_\alpha$. It is then true that $Y_\mathfrak{m}$ is a relatively regular element in $M$. We have that $$\chi(\overline{u})=\psi(\tr(Y\log(\overline{u}))),\,\overline{u}\in \overline{U},$$ restricted to $M\cap \overline{U}$ is equal to $\chi^M$, where $\chi^M(\overline{u})=\psi(\text{tr}_\mathfrak{m}(Y_\mathfrak{m}\log(\overline{u}))),\text{ for }\overline{u}\in M\cap\overline{U}$. Indeed, this follows from the fact that the bilinear form $\tr$ restricted to $\mathfrak{m}\times \mathfrak{m}$ is equal to $\text{tr}_{\mathfrak{m}}$ \cite[Lemma 5]{McNinch}, and from the fact that $Y-Y_\mathfrak{m}\in m^{\perp}$ with respect to $\tr$.   
Note that for a co-character $\varphi \in X_{\ast}(A)$ and a relatively regular nilpotent element $Y$, such that $\text{Ad}(\varphi(c))Y=c^{2}Y$, for $c\in \overline{F}$ we get $\text{Ad}(\varphi(c))Y_{\mathfrak{m}}=c^{2}Y_\mathfrak{m}$, so it makes sense to write $(Y_{\mathfrak{m}},\varphi)$.

\begin{definition} Let $(\sigma,W)$ be an irreducible representation of $M$. We say that $(\sigma,W)$ is \emph{$(Y,\varphi)$--generic} if $(\sigma,W)$ is $(Y_\mathfrak{m},\varphi)$--generic.  
\end{definition}

In order to generalize local coefficients we will need two hypothesis to be satisfied, we call them $\mathbf{H}_1$ and $\mathbf{H}_2$. 

\begin{itemize}
\item[$\mathbf{H}_1.$]  We can find a lattice $\mathcal{L}$ of $\mathfrak{g}$ that satisfies properties 1)--2) of Lemma \ref{lattice} with respect to $B_Y$, such that for any standard parabolic subgroup $P=MN$ of $G$ we get that $\mathcal{L}_\mathfrak{m}=\mathfrak{m}\cap \mathcal{L}$ satisfies properties 1)--2) of Lemma \ref{lattice} with respect to $B_{Y_\mathfrak{m}}$. Moreover, we can choose $\mathcal{L}$ to also satisfy $^w\mathcal{L}=\mathcal{L}$, for some full set of representatives of $\widetilde{w}\in {_GW}$. \\ 

\item[$\mathbf{H}_2.$] Let $(\sigma,W)$ be a $(Y_\mathfrak{m},\varphi)$--generic representation of $M$. Then dim$_{\mathbb{C}}W_{Y_\mathfrak{m},\varphi}$=dim$_{\mathbb{C}}I(0,\sigma)_{Y,\varphi}$. In other words the dimension of the space of degenerate Whittaker forms is invariant under induction. 
\end{itemize}

Let $\mathcal{L}$ be a lattice satisfying $\mathbf{H}_{1}$. We get out of Theorem \ref{G_n} a sequences of subgroups $G_n=\exp(\varpi^n\mathcal{L})$ and characters on $\chi_n$ that depend on $Y$. Also out of Theorem \ref{G_n}, we have a sequence of subgroups $M_n=\exp(\varpi^n\mathcal{L}_\mathfrak{m})$ and characters $\chi_n^M$ depending on $Y_\mathfrak{m}$. We see that $G_n\cap M=M_n$ and $\chi_n$ restricted to $M_n$ is equal to $\chi_n^M$. We have that if $G$ is quasi-split the dimension of the space of Whittaker functionals does not change for the induced representation \cite[3.4.6]{ShahidiBook}, therefore $\mathbf{H}_{2}$ is always satisfied in this case. We assume throughout this section that $\mathbf{H}_1$ and $\mathbf{H}_2$ are satisfied. We abuse the notation and write $\chi_n$ for the character $G_n$ as well as for the restriction to $M_n$.  We are going to assume from now on that every time we choose a representative $w$ of an element of $\tilde{w}\in _GW$, is one that satisfies the condition $^w\mathcal{L}=\mathcal{L}$. 

\subsection{Construction of some special functions}
Let $(\sigma,W)$ a $(Y,\varphi)$--generic representation of $M$. Let $v$ be an element in $W$ such that $\sigma(m)v=\chi_n(m)v$ for $m\in M_n$. We define $f_{(\nu,\sigma,v)} \in I(\nu,\sigma)$ to be the function with support in $P\overline{N}_n$ given by $f(pj)=\sigma(p)\delta_P(p)^{-1/2}\nu(p)\chi_n(j)f(1)$, for $p\in P, j\in \overline{N}_n$; where $f(1)=v$, $\delta_P$ is the modular function for $P$ and $\nu$ is an unramified character of $M$ extended to be trivial on $N$. 

\begin{prop} \label{functions}The function $f_{(\nu,\sigma,v)}$ is in $ I(\nu,\sigma)^{G_n,\chi_n} $, that is $f_{(\nu,\sigma,v)}(gx)=\chi_{n}(x)f_{(\nu,\sigma)}(g)$, for all $g\in G$, $x\in G_n$.
\end{prop}
\begin{proof}

Let us denote $f_{(\nu,\sigma,v)}$ by $f$ for short. We want to show that $f(gx)=\chi_n(x)f(g)$, for $g\in G$, $x\in G_n$. We have that $G_n=\overline{N}_n P_n$, so it is enough to consider two cases, the case $x\in \overline{N}_n$, and the case $x\in P_n$. The case $x\in \overline{N}_n$ follows right out of the definition of $f$, so it is enough to consider the case  $x\in P_n$. 
We have that $Q\supset N$ and that $\chi_n$ is a character of $G_n$ trivial on $Q_n$, thus also trivial on $N_n$. If $x\in P_n=M_nN_n$, we can write $x=my$ where $m\in M_n$ and $y\in N_n$. Then \begin{align}\label{x in P}f(x)=f(my)=\sigma(m)v=\chi_n(m)v=\chi_n(x)v.\end{align} Let $g\in P\overline{N}_n$ and $x\in P_n=M_nN_n$. We have $g=pj$, for some $p\in P$ and some $j\in \overline{N}_n$. The element $x^{-1}j^{-1}xj^{-1}\in \ker \chi_n$, so there exists a $z\in \ker \chi_n$ such that $x^{-1}jx=zj$. We write $z=z_1z_2$ where $z_1\in P_n$ and $z_2\in \overline{N}_n$. We then get 
\begin{align*}
f(gx)=&f_{(\nu,\sigma)}(pjx)=f(pxx^{-1}jx)=f(pxzj)=f(pxz_1z_2j)\\
=&\sigma(p)\sigma(xz_1)\delta^{-1/2}_p(pxz_1)\nu(pxz_1)\chi_n(z_2j)v\\
=&\sigma(p)\delta^{-1/2}_P(p)\nu(p)\chi_n(z_2j)f(xz_1)\\
=&\sigma(p)\delta^{-1/2}_P(p)\nu(p)\chi_n(z_2j)\chi_n(xz_1)v \text{\, (by equation \ref{x in P})}\\
=&\sigma(p)\delta^{-1/2}_P(p)\nu(p)\chi_n(xz)\chi(j)v\\ 
=&\chi_n(xz)f(pj)=\chi(x)f(g)
\end{align*}
We get that $f(gx)=\chi_n(x)f(g)$ is not zero for $x\in G_n$, and $g \in P\overline{N}_n$. Therefore $P\overline{N}_nG_n=P\overline{N}_n$. We get that if $g\not\in P\overline{N}_n$ then $gx\not\in P\overline{N}_n$, which implies that $f(gx)=0=\chi_n(x)f(g)$. We conclude that $f(gx)=\chi_n(x)f(g)$, for all $g\in G$.
\end{proof}
Note that the functions $\text{Res}^P_Kf_{(\nu,\sigma,v)}$ are independent of $\nu$. Let $\mathcal{B}=\left\lbrace v_1,v_2\ldots ,v_k \right\rbrace$ be a basis for $W_n$. We then have that $$\mathcal{B}(\nu,\sigma)=\left\lbrace f_{(\nu,\sigma,v_1)},f_{(\nu,\sigma,v_2)}\ldots ,f_{(\nu,\sigma,v_k)} \right\rbrace$$ is linearly independent because their evaluation at the identity is linearly independent. We then conclude thanks to $\mathbf{H_2}$ that $\left\lbrace f_{(\nu,\sigma,v_1)},f_{(\nu,\sigma,v_2)}\ldots ,f_{(\nu,\sigma,v_k)} \right\rbrace$ is a basis for $ I(\nu,\sigma)^{G_n,\chi_n}$.

In order to define generalized local coefficients it is natural to assume some relation from $\chi$ and $w$ coming from the compatibility condition in the quasi-split case. The character $\chi$ only depends on $Y$ and is more convenient to define the compatibility condition in terms of $w$ and $Y$. 
\begin{definition} Let $\tilde{w}\in {_GW}$. We say that $w$ and $(Y,\varphi)$ are \emph{compatible} if $(Y^w)_\mathfrak{m}=Y_\mathfrak{m}$, for  $\,\overline{u}\in M\cap\overline{U}$.  
\end{definition}
Let us suppose then that $w$ and $(Y,\varphi)$ are compatible. Using $\mathbf{H}_2$ again, we have $w^{-1}G_nw=G_n$. We claim that compatibility of $w$ and $Y$ implies $^w\chi_n(x)=\chi_n(x),\,x\in M'_n$. Indeed, take $x\in M'_n$ and write $x=qj$, for $q\in Q\cap M'_n\,, j\in \overline{U}\cap M'_n$. Since $\tilde{w}(\theta)=\theta'\subset \Delta$, we get that $w^{-1}qw \in Q\cap M_n$ and $w^{-1}jw\in \overline{U}\cap M_n$.  \begin{align*}^w\chi_n(x)=&\chi_n(w^{-1}qjw)=\chi_n(w^{-1}jw)=\psi(\tr(\varpi^{-2n}Y\log(w^{-1}jw))\\
=&\psi(\text{tr}_{\mathfrak{m}}(\varpi^{-2n}(Y^w)_\mathfrak{m}\log(j)))=\psi(\text{tr}_{\mathfrak{m}}(\varpi^{-2n}(Y)_\mathfrak{m}\log(j)))\\
=&\psi(\tr(\varpi^{-2n}Y\log(j))=\chi_n(j)=\chi_n(x)\\
\end{align*} 
Given $v\in W_n=(\sigma,W)^{M_n,\chi_n}$, we get for $x\in M'_n$ that $$^w\sigma(x)v=\sigma(w^{-1}xw)v=\chi_n(w^{-1}xw)v=\chi_n(x)v$$ therefore $v\in (^w\sigma,W)^{M'_n,\chi_n}$. We get out of the compatibility condition that $\chi(u)=\chi(wuw^{-1}),$ for $u\in \overline{U}\cap M$ and thus the space spanned by $$\lbrace \sigma(u)v-\chi(u)v\rbrace,\, u\in \overline{U}\cap M,\,v\in W$$ is equal to the space spanned by $$\lbrace {^w\sigma}(u)v-{^w\chi}(u)v\rbrace,\, u\in \overline{U}\cap M',\,v\in W$$ and equal the one spanned by $$\lbrace {^w\sigma}(u)-{\chi}(u)\rbrace,\, u\in \overline{U}\cap M', \,v\in W.$$ Since taking the quotient of $W$ by any of these subspaces is not zero we conclude that $(^w\sigma,W)$ is $(Y,\varphi)$--generic. We then can construct the function $f_{(^w\nu,{^w\sigma},\,v)}$, using Proposition \ref{functions} and we get that $f_{(^w\nu,{^w\sigma},\,v)}$ is in $ I(^w\nu,{^w\sigma})^{G_n,\chi_n}$. The same reasoning shows that $$\mathcal{B}(^w\nu, {^w\sigma})=\left\lbrace f_{({^w\nu},{^w\sigma},\,v_1)},f_{({^w\nu},{^w\sigma},\,v_2)}\ldots ,f_{({^w\nu},{^w\sigma},\,v_k)} \right\rbrace$$ is a basis for $ I(^w\nu,{^w\sigma})^{G_n,\chi_n}$. The operator $A(\nu,\sigma,w)$ maps $I(\nu,{\sigma})^{G_n,\chi_n}$ isomorphically onto $I(^w\nu,{^w\sigma})^{G_n,\chi_n}$. We denote by $[A(\nu,\sigma,w]_{\mathcal{B}}$ the matrix representation of the operator $A(\nu,\sigma,w)$ when restricted to $I(\nu,\sigma)^{G_n,\chi_n}$, with respect to the bases $\mathcal{B}(\nu,\sigma)$ and $\mathcal{B}(^w\nu,{^w\sigma})$. We contend that $[A(\nu,\sigma,w]_{\mathcal{B}}$ is a scalar matrix, this scalar is what we are after, i.e. a generalized local coefficient. We first prove the case where $\sigma$ is unitary.

\begin{prop}\label{unitaryscalar} Let $(\sigma,W)$ be a unitary, irreducible, $(Y,\varphi)$-generic representation of $M$. Let $\mathcal{B}$ be any basis for $W_n$ and suppose that $w$ is compatible with $(Y,\varphi)$. Then for sufficiently large $n$, $[A(\nu,\sigma,w)]_{\mathcal{B}}$ is a scalar matrix.  
\end{prop}

\begin{proof}  Let us denote by $\left\langle\,,\right\rangle$ the duality between $V(\nu,\sigma)$ and $V(-\overline{\nu},\sigma)$ defined in \ref{duality}. Let $(\,,)$ be a non--degenerate Hermitian form on the unitary representation $(\sigma,W)$. Let $K$ be a maximal compact subgroup such that $G=KP=KP'$. Let us suppose that $\mathcal{B}=\left\lbrace v_1,v_2,\ldots ,v_k\right\rbrace$ is an orthonormal basis for $W_n$. We have for $1\leq i,j\leq k$, \begin{align}\left\langle f_{(\nu,\sigma,v_i)},f_{(-\overline{\nu},\sigma,v_j)}\right\rangle=&\int\limits_{K}(f_{(\nu,\sigma,v_i)}(x),f_{(-\overline{\nu},\sigma,v_j)}(x))dx\notag\\
=&\mu_K(P\overline{N}_n\cap K)\delta_{ij}=\begin{cases} \mu_K(P\overline{N}_n\cap K) &\mbox{if } i=j \\
0& \text{otherwise.} \end{cases}\label{selfdual}\end{align}

Let us denote as well by $\left\langle\,,\right\rangle$ the duality between $I(^w\nu,{^w\sigma})$ and $I(^w(-\overline{\nu}),{^w\sigma})$, it should cause no confusion. We similarly get that \begin{align}\left\langle f_{(^w\nu,{^w\sigma},v_i)},f_{(^w(-\overline{\nu}),{^w\sigma},v_j)}\right\rangle=&\int\limits_{K}( f_{(^w\nu,{^w\sigma},v_i)}(x),f_{(^w(-\overline{\nu}),{^w\sigma},v_j)}(x))dx \notag\\
=&\mu_K(P'\overline{N'}_n\cap K)\delta_{ij}=\begin{cases} \mu_K(P'\overline{N'}_n\cap K) &\mbox{if } i=j \\
0& \mbox{if } \text{otherwise.} \end{cases}\label{selfdual2}\end{align}
We denote by $A(\nu,\sigma,w)^*$ the operator from $I(^w\nu,{^w\sigma})$ into $I(\nu,\sigma)$ satisfying \begin{align}\left\langle A(\nu,\sigma,w)f_{1},f_{2}\right\rangle =\left\langle f_{1}, A(\nu,\sigma,w)^*f_{2}\right\rangle, \label{conjugate matrix}\end{align}
for $f_{1}\in I(^w\nu,{^w\sigma}),\, f_{2}\in I(^w(-\overline{\nu}),{^w\sigma})$. Let $[A(\nu,\sigma,w)]_{\mathcal{B}}=(a_{i,j}(\nu))$, and let $[A(\nu,\sigma,w)^*]_{\mathcal{B}}=(b_{i,j}(\nu))$, for $1\leq i,\,j\leq k$. Let $r_1=\mu_K(P\overline{N}_n\cap K)$, and $r_2=\mu_K(P'\overline{N'}_n\cap K)$. We then get from equations \ref{selfdual}, \ref{selfdual2} and \ref{conjugate matrix}, that $r_1a_{i,j}(\nu)=r_2\overline{b_{j,i}(\nu)}$. In other words the matrix $[A(\nu,\sigma,w)^*]_{\mathcal{B}}$ is the complex conjugate transpose of the matrix $[A(\nu,\sigma,w)]_{\mathcal{B}}$ times a positive real number $r=r_1/r_2$. We have from Proposition 5.2.1 in \cite{ShahidiBook}, that $$A(\nu,\sigma,w)^*=A(^w(-\overline{\nu}),{^w\sigma},w^{-1}).$$
Suppose that $\nu \in i\mathfrak{a}^*$, then $\nu=-\overline{\nu}$. We have that $A(\nu,\sigma,w)A(^w(\nu),{^w\sigma},w^{-1})$ is a scalar times the identity. We therefore get that $[A(\nu,\sigma,w)]_{\mathcal{B}}$ commutes with $[A(^w(\nu),{^w\sigma},w^{-1})]_{\mathcal{B}}=[A(\nu,\sigma,w)^*]_{\mathcal{B}}$, which implies that $[A(\nu,\sigma,w)]_{\mathcal{B}}$ commutes with its complex conjugate transpose. We therefore get in this case that $[A(\nu,\sigma,w)]_{\mathcal{B}}$ and $[A(^w(\nu),{^w\sigma},w^{-1})]_{\mathcal{B}}$ are normal matrices and therefore simultaneously diagonalizable, since their product is a scalar matrix, each must be a scalar matrix. We get that $[A(\nu,\sigma,w)]_{\mathcal{B}}$ is a scalar matrix for any $\nu$ by analytic continuation. This finishes the proof.  
\end{proof}

\begin{thm}\label{scalar} Let $(\sigma,W)$ be an irreducible, $(Y,\varphi)$-generic representation of $M$. Let be $\mathcal{B}$ any basis for $W_n$ and suppose that $w$ is compatible with $(Y,\varphi)$. Then $[A(\nu,\sigma,w)]_{\mathcal{B}}$ is a scalar matrix, for sufficiently large $n$.  
\end{thm}
\begin{proof}
 We have that $\sigma$ is irreducible and thus is admissible. Then the contragredient representation $\check{\sigma}$ is irreducible as well, otherwise $\check{\check{\sigma}}\cong \sigma$ would be reducible and this is not the case. We get by the Langlands' quotient theorem (Theorem 4.1 (3) in \cite{Silberger}) that there exists a standard parabolic subgroup $P_{\ast}=M_{\ast}N_{\ast} \subset P$, an irreducible tempered representation $(\sigma_{\ast}, W_{\ast})$ of $M_{\ast}$, and a $\nu_{\ast}\in \mathfrak{a}^*_{\ast,\mathbb{C}}=X^*(M_{\ast})\otimes \mathbb{C}$ such that $$ I_{M}(\nu_{\ast},\sigma_{\ast})\longrightarrow \check{\sigma} \longrightarrow 0,$$
is exact. Where $I_{M}(\nu_{\ast},\sigma_{\ast})$ is the representation of $M$ obtained by normalized parabolic induction from $P_{\ast}\cap M$ to $M$. Taking contragerdients of the exact sequence above and using the fact that the contragredient of $I_{M}(\nu_{\ast},\sigma_{\ast})$ is $I_{M}(-\nu_{\ast},\check{\sigma}_{\ast})$ we get the exact sequence $$0\longrightarrow  \sigma \longrightarrow I_{M}(-\nu_{\ast},\check{\sigma}_{\ast}).$$
Since $\sigma_{\ast}$ is unitary we get that $\check{\sigma}_{\ast}$ is unitary as well. 

To avoid working with $\check{\sigma}_{\ast}$ and $-\nu_{\ast}$, we redefine $(\sigma_{\ast}, W_{\ast})$ and $\nu_{\ast}$ in a more convenient way. We let $(\sigma_{\ast}, W_{\ast})$ be a unitary representation of $M_{\ast}$ such that there exists an injective morphism $$T: \sigma \longrightarrow I_{M}(\nu_{\ast},\sigma_{\ast}),$$ for some $\nu_{\ast}\in \mathfrak{a}^*_{\ast,\mathbb{C}}=X(M_{\ast})\otimes \mathbb{C}$. We have shown above that such a $(\sigma_{\ast}, W_{\ast})$ and $\nu_{\ast}$ do exist. We denote by $V_M(\nu_{\ast},\sigma_\ast)$ the space of functions where $I_{M}(\nu_{\ast},\sigma_{\ast})$ acts by right translation.  

Given $\nu\in \mathfrak{a}^*_{\mathbb{C}} ,$ we let $\tilde{\nu} \in \mathfrak{a}^*_{\ast,\mathbb{C}} $ to satisfy $$\left\langle \tilde{\nu}, H_{M_\ast}(m_{\ast}) \right\rangle= \left\langle \nu, H_{M}(m_{\ast}) \right\rangle,\,\forall \, m_{\ast}\in M_{\ast}. $$ Using the same notation than Lemma \ref{commutativity}, we have $$A(\tilde{\nu}+\nu_{\ast},\sigma_{\ast},w)(T_{\ast}(\nu)(f))=T^w_{\ast}(\nu)(A(\nu,\sigma,w)f),\,\forall f\in V(\nu,\sigma).$$
Since $(\sigma,w)$ is $(Y,\varphi)$--generic, we get that $W_{Y_{\mathfrak{m}},\overline{U}\cap M}\neq 0$, and by the exactness of the twisted Jaquet functor we must have $V_{M}(\nu_{\ast},\sigma_{\ast})_{Y_{\mathfrak{m}},\overline{U}\cap M}\neq 0$. We therefore have by hypothesis $\mathbf{H}_2$ that $(\sigma_{\ast},W_\ast)$ is a $(Y,\varphi)$--generic. Let $(W_\ast)_n$ be the $\chi_n$ isotypic component of $(\sigma_{\ast},W_\ast)$ when restricted to $(M_\ast)_n$, and let $\mathcal{B}_{\ast}$ be a basis for $(W_\ast)_n$. We then get by Proposition \ref{unitaryscalar}, that for sufficiently large $n$, $$[A(\tilde{\nu}+\nu_{\ast},\sigma_{\ast},w)]_{\mathcal{B}_{\ast}}$$ is a scalar matrix. We then have that there exists a $D(\nu)^{-1}\in \mathbb{C}\cup \{\infty\}$, such that for all $v_{\ast}\in (W_{\ast})_n$ $$A(\tilde{\nu}+\nu_{\ast},\sigma_{\ast},w)f_{(\tilde{\nu}+\nu_{\ast},\sigma_{\ast},v_{\ast})}=D(\nu)^{-1}f_{({^w\tilde{\nu}}+{^w\nu_{\ast}},{^w\sigma_{\ast}},v_{\ast})}$$ 
Let $\mathcal{B}$ be a basis for $W_n$, and let $v\in \mathcal{B}$. We have that $T_{\ast}(\nu)(f_{(\nu,\sigma,v)})$ is a non-zero vector in the $\chi_n$ isotypic component of $I(\nu_{\ast}+\tilde{\nu},\sigma_\ast)$ when restricted to $G_n$. Therefore we have that the support of $T_{\ast}(\nu)(f_{(\nu,\sigma,v)})$ is equal to $P_\ast (\overline{N_{\ast}})_n$. We have that 
$$T_{\ast}(\nu)(f_{(\nu,\sigma,v)})(1)=T(f_{(\nu,\sigma,v)}(1))(1)=T(v)(1)$$
We conclude that  $$T_{\ast}(\nu)(f_{(\nu,\sigma,v)})=f_{(\nu_{\ast}+\tilde{\nu},\,\sigma_{\ast},\,T(v)(1))}$$
Similarly we have that $$T^w_{\ast}(\nu)(f_{({^w\nu},{^w\sigma},v)})=f_{({^w\nu_{\ast}}+{^w\tilde{\nu}},\,{^w\sigma}_{\ast},\,T^w(v)(1))}$$ 
We note that $T(v)(1)=T^w(v)(1)$, putting it all together we get 
\begin{align*} T^w_{\ast}(\nu)(A(\nu,\sigma,w)f_{({\nu},{\sigma},v)})=&A(\tilde{\nu}+\nu_{\ast},\sigma_{\ast},w)T_{\ast}(\nu)(f_{({\nu},{\sigma},v)})=A(\tilde{\nu}+\nu_{\ast},\sigma_{\ast},w)f_{(\nu_{\ast}+\tilde{\nu},\,\sigma_{\ast},\,T(v)(1))}\\
=&D(\nu)^{-1}f_{({^w\tilde{\nu}}+{^w\nu_{\ast}},{^w\sigma_{\ast}},T(v)(1))}=D(\nu)^{-1}T^w_{\ast}(\nu)(f_{({^w\nu},{^w\sigma},v)})
\end{align*}

Using the fact that the linear map $T^w_{\ast}(\nu)$ is injective, we get $$A(\nu,\sigma,w)f_{({\nu},{\sigma},v)}=D(\nu)^{-1}f_{({^w\nu},{^w\sigma},v)}.$$ 
Therefore $[A(\nu,\sigma,w)]_{\mathcal{B}}$ is a scalar matrix. This finishes the proof of the theorem. 
\end{proof}

\begin{definition} Let $(\sigma,W)$ be an irreducible, $(Y,\varphi)$-generic representation of $M$. Let $\mathcal{B}$ be any basis for $W_n$ and suppose that $w$ is compatible with $(Y,\varphi)$. We know by Theorem \ref{scalar} that $[A(\nu,\sigma,w)]_{\mathcal{B}}$ is a scalar matrix, for sufficiently large $n$. We denote by $D_{(Y,\varphi,n)}(\nu,\sigma,w)$ the meromorphic function that satisfies $D_{(Y,\varphi,n)}(\nu,\sigma,w)[A(\nu,\sigma,w)]_{\mathcal{B}}$ is the identity matrix. We call $D_{(Y,\varphi,n)}(\nu,\sigma,w)$ a \emph{generalized local coefficient}. 
\end{definition}

We have that generalized local coefficients are related to the Plancherel measure in the way we expect from Proposition \ref{Plancherel measures}.
\begin{cor} \label{gral. Plancherel measure} We have for $n$ sufficiently large $$\mu(\nu,\sigma,w)\gamma_w(G/P)\gamma_{w^{-1}}(G/P')=D_{(Y,\varphi,n)}(\nu,\sigma,w)^{-1}D_{(Y,\varphi,n)}({^w\nu},{^w\sigma},w^{-1})^{-1}$$ 
\end{cor}
\begin{proof} Take $v\in W_n$ and constructing the function $f_{(\nu,\sigma,v)}$. We then have that \begin{align*}&\mu(\nu,\sigma,w)\gamma_w(G/P)\gamma_{w^{-1}}(G/P')f_{(\nu,\sigma,v)}=A({^w\nu},{^w\sigma},w^{-1})A(\nu,\sigma,w)f_{(\nu,\sigma,v)}\\
&=D_{(Y,\varphi,n)}(\nu,\sigma,w)^{-1}A({^w\nu},{^w\sigma},w^{-1})f_{({^w\nu},{^w\sigma},v)}=D_{(Y,\varphi,n)}(\nu,\sigma,w)^{-1}D_{(Y,\varphi,n)}({^w\nu},{^w\sigma},w^{-1})^{-1}f_{(\nu,\sigma,v)}\end{align*}

\end{proof}

We have that generalized local coefficients are multiplicative. 

\begin{cor} \label{multiplicativity of geral.local} Let $(\sigma,W)$ be $(Y,\varphi)$--generic representation of $M$. Let $\tilde{w}, \tilde{w}_1,\ldots \tilde{w}_k $ be as in Theorem \ref{factorization of intertwining}. Suppose that $w,w_1,\ldots ,w_k$ are compatible with $(Y,\varphi)$. Let $\nu_1=\nu$, $\nu_i={^{w_{i-1}}\nu_{i-1}}$, $\sigma_1=\sigma$ and $\sigma_i={^{w_{i-1}}\sigma_{i-1}},$ for $2\leq i\leq k-1$. Then 
for $n$ sufficiently large $$D_{(Y,\varphi,n)}(\nu,\sigma,w)=\prod\limits_{i=1}^{n-1}D_{(Y,\varphi,n)}(\nu_i,\sigma_i,w_i).$$
\end{cor}

\begin{proof} We have out of Theorem \ref{factorization of intertwining}. That $$A(\nu,\sigma,w)=A(\nu_{k-1},\sigma_{k-1},w_{k-1})\cdots A(\nu_1,\sigma_1,w_1).$$ 
Take $v\in W_n$ and construct the functions $f_{(\nu_i,\sigma_i,v)}$, $1\leq i\leq k-1$. We then have that $$A(\nu_{i-1},\sigma_{i-1},w_{i-1})f_{(\nu_{i-1},\sigma_{i-1},v)}=D_{(Y,\varphi,n)}(\nu_{i-1},\sigma_{i-1},w_{i-1})^{-1}f_{(\nu_i,\sigma_i,v)}$$
We then get
\begin{align*}&D_{(Y,\varphi,n)}(\nu,\sigma,w)^{-1}f_{({^w\nu},{^w\sigma},v)}=A(\nu,\sigma,w)f_{(\nu,\sigma,v)}\\
&=A(\nu_{k-1},\sigma_{k-1},w_{k-1})\cdots A(\nu_1,\sigma_1,w_1)f_{(\nu_1,\sigma_1,v)}=\prod\limits_{i=1}^{n-1}D_{(Y,\varphi,n)}(\nu_i,\sigma_i,w_i)f_{({^w\nu},{^w\sigma},v)}
\end{align*}
The theorem follows.
\end{proof}

\subsection{Relation to Shahidi's local coefficients}\label{relation to local coeff.}

In this subsection we restrict to the case where $G$ is quasi-split and use the notation and assumptions needed to define local coefficients defined in Subsection \ref{local coefficients}. 
Let $-\Delta$ denote the negative of the simple roots in $\Delta$. There exists a relatively $(-\Delta)$--regular element $X\in \overline{\mathfrak{u}}$, such that, $$\chi(u)=\psi \circ \tr(X\log(u)),\, u\in U.$$

\begin{lem}\label{compatibility conditions} Let $X$ be as above. The conditions that $\chi$ is compatible with $w$, $w_0$ and $w'_0$ implies the conditions $$(X^w)_\mathfrak{m}=X_{\mathfrak{m}},\,(X^{w_0})_\mathfrak{m}=X_{\mathfrak{m}},\,(X^{w'_0})_{\mathfrak{m}'}=X_{\mathfrak{m}'}.$$
\end{lem} 

\begin{proof}We show $(X^w)_\mathfrak{m}=X_{\mathfrak{m}}$ and analogous arguments will work for $w_0$ and $w'_0$. Indeed, the fact that $w$ is compatible with $\chi$ means that $$\psi\circ \tr(X\log(u))=\chi(u)=\chi(wuw^{-1})=\psi\circ\tr(X\log(wuw^{-1}))=\psi\circ\tr(X^w\log(u))$$
for $u\in U\cap M$. We then get that
\begin{align*}\psi\circ \text{tr}_{\mathfrak{m}}(X_{\mathfrak{m}}\log(u))&=\psi\circ\text{tr}_{\mathfrak{m}}((X^w)_{\mathfrak{m}}\log(u))\implies \text{tr}_{\mathfrak{m}}((X-X^w)_{\mathfrak{m}}Z)\in \mathcal{O}, \forall \, Z\in \overline{\mathfrak{u}}\cap \mathfrak{m}.\\
\implies &\text{tr}_{\mathfrak{m}}((X-X^w)_{\mathfrak{m}}Z)=0, \forall \, Z\in \overline{\mathfrak{u}}\cap \mathfrak{m}.
\end{align*}
Using the fact that the bilinear form $\text{tr}_{\mathfrak{m}}$ is non-degenerate and that the dual of the Lie algebra of $U\cap M$ is the Lie algebra of $\overline{U}\cap M$, we deduce that $(X^w)_\mathfrak{m}=X_{\mathfrak{m}}$. Similar arguments work to show that $(X^{w_0})_\mathfrak{m}=X_{\mathfrak{m}},\,(X^{w'_0})_{\mathfrak{m}'}=X_{\mathfrak{m}'}$. 
\end{proof}

\begin{lem}\label{X,Y} Let $Y=w_{\ell}^{-1}X w_{\ell}={X^{w_{\ell}}}$. Then $(Y,\varphi)$ is compatible with $w$, i.e. $(Y^w)_\mathfrak{m}=Y_\mathfrak{m}$.
\end{lem}

\begin{proof} We get out of $(X^{w_0})_\mathfrak{m}=X_{\mathfrak{m}}$ that 
\begin{align} ({^{w_{\ell}^{M}}X)_{\mathfrak{m}}}=(X^{w_{\ell}})_\mathfrak{m}=Y_{\mathfrak{m}}\label{m}&
\end{align}
 Using now $(X^{w'_0})_\mathfrak{m'}=X_{\mathfrak{m'}}$ we see  
\begin{align*} (X^{w_{\ell}ww_{\ell}^{M}w^{-1}})_{\mathfrak{m'}}=X_\mathfrak{m'} \implies &(X^{w_{\ell}ww_{\ell}^{M}})_{\mathfrak{m}}=(X^w)_\mathfrak{m}=X_\mathfrak{m}\\
\implies (X^{w_{\ell}w})_{\mathfrak{m}}=({^{w_{\ell}^{M}}X})_\mathfrak{m} \implies & \text{ using \ref{m} we get } (Y^w)_\mathfrak{m}=Y_\mathfrak{m}.  
\end{align*}  
This finishes the proof of the lemma.
\end{proof} 

\begin{lem}\label{generic} Let $Y$ be as in Lemma \ref{X,Y}. We have $\chi^{w_{\ell}}(u)=\psi\circ \tr(Y\log(u))$, for $u\in \overline{U}$, and $^{w_{\ell}^M}\chi(u)=\psi\circ \text{tr}_{\mathfrak{m}}(Y_{\mathfrak{m}}\log(u))$, for $u\in \overline{U}\cap M$. Moreover, if $(\sigma,W)$ is an irreducible $\chi$--generic representation of $M$, then $(\sigma,W)$ is $(Y,\varphi)$--generic. 
\end{lem}
\begin{proof} We have for $u\in \overline{U}$ that \begin{align*} \chi^{w_{\ell}}(u)=&\psi\circ\tr(X\log(w_{\ell}uw_{\ell}^{-1}))=\psi\circ\tr(X^{w_{\ell}}\log(u))=\psi\circ \tr(Y\log(u))\\
\end{align*} 
We have for  $u\in \overline{U}\cap M$ \begin{align*} ^{w_{\ell}^M}\chi(u)=&\psi\circ\tr(X\log((w_{\ell}^M)^{-1}uw_{\ell}^{M}))=\psi\circ\text{tr}_{\mathfrak{m}}((^{w_{\ell}^M}X)_{\mathfrak{m}}\log(u))=\psi\circ\text{tr}_{\mathfrak{m}}(Y_{\mathfrak{m}}\log(u))\\
\end{align*} 
 In order to prove that then $(\sigma,W)$ is $(Y,\varphi)$--generic, we need to show that $W_{Y_\mathfrak{m},\overline{U}\cap M}\neq 0$. We have that the space $\Hom_{U\cap M}(\sigma,\chi)$ is one dimensional, fix then $\lambda$ a Whittaker functional in this space. We have a map from $$\Hom_{U\cap M}(\sigma,\chi)\longrightarrow \Hom_{\overline{U}\cap M}(\sigma,{^{w_{\ell}^{M}}\chi}),$$ given by $\lambda\mapsto \lambda \circ \sigma({w_{\ell}^{M}})^{-1}$. We therefore have that $\Hom_{\overline{U}\cap M}(\sigma,{^{w_{\ell}^{M}}\chi})\neq 0$ and thus $W_{Y_\mathfrak{m},\overline{U}\cap M}\neq 0.$\end{proof} 
Let $t=\varphi(\varpi)$. We define $\tilde{\chi}$ by $$\tilde{\chi}(u)=\chi(t^{-n}ut^{n})=\psi\circ \tr(\varpi^{-2n}X\log(u)),\, u\in U.$$ An easy computation shows that \begin{align}\tilde{\chi}(u)=\chi(w_{\ell}t^{n}w_{\ell}^{-1}uw_{\ell}t^{-n}w_{\ell}^{-1}),\,u\in U\label{t}\end{align}
Indeed, we have that  
\begin{align*}\tilde{\chi}(u)=&\psi\circ \tr(\varpi^{-2n}X\log(u))=\psi\circ \tr(\varpi^{-2n}X^{w_\ell}\log(w_{\ell}^{-1}uw_{\ell}))\\
=&\psi\circ \tr(\varpi^{-2n}Y\log(w_{\ell}^{-1}uw_{\ell}))=\psi\circ \tr(t^{-n}Yt^n\log(w_{\ell}^{-1}uw_{\ell}))\\
=&\psi\circ \tr(Y\log(t^nw_{\ell}^{-1}uw_{\ell}t^{-n}))=\psi\circ \tr(X\log(w_{\ell}t^nw_{\ell}^{-1}uw_{\ell}t^{-n}w_\ell))\\
=&\psi\circ \tr(X^{w_{\ell}}\log(t^nw_{\ell}^{-1}uw_{\ell}t^{-n}))=\chi(w_{\ell}t^{n}w_{\ell}^{-1}uw_{\ell}t^{-n}w_{\ell}^{-1})\end{align*}

Since $(\sigma, W)$ is $\chi$--generic we get that $({^{t^n}\sigma},W)$ is $\tilde{\chi}$--generic, and since $({^{t^n}\sigma},W)$ is isomorphic to $(\sigma,W)$, we conclude that $(\sigma,W)$ is $\tilde{\chi}$--generic. We claim that $\tilde{\chi}$ is compatible with $w$, $w_0$ and $w'_0$. We show that $\tilde{\chi}$ is compatible with $w$ and analogous arguments will work for $w_0$ and $w'_0$. Let $u\in U\cap M$, then  
\begin{align*}{\tilde{\chi}}^w(u)=&\psi(\tr(\varpi^{-2n} X\log(wuw^{-1}))= \psi(\varpi^{-2n}\tr(Xw\log(u)w^{-1}))\\
=&\psi(\varpi^{-2n}\tr(X^w\log(u)))=\psi(\varpi^{-2n}\text{tr}_{\mathfrak{m}}((X^w)_\mathfrak{m}\log(u)))\\
=&\psi(\varpi^{-2n}\text{tr}_{\mathfrak{m}}(X_\mathfrak{m}\log(u)))=\psi(\varpi^{-2n}\tr(X\log(u)))\\
=&\tilde{\chi}(u)
\end{align*}
We now state and prove a theorem that relates local coefficients to generalized local coefficients. 

\begin{thm} Let $(\sigma,W)$ be $\chi$--generic representation of $M$. Suppose that $\chi$ is compatible with $w$, $w_0$ and $w_0'$. Let $\chi(u)=\psi \circ \tr(X\log(u)),\,u \in U,$ for some relatively $(-\Delta)$--regular element $X$ i.e. of the form $X=\sum_{\alpha\in \Delta} X_{-\alpha}$, $X_{-\alpha}\in \mathfrak{g}_{-\alpha}$. Let $Y=X^{w_{\ell}}$. Then for $n$ sufficiently large we get $$C_{\tilde{\chi}}(\nu,\sigma,w)=\frac{\mu_{\overline{N}}(\overline{N}_{n})}{\mu_{\overline{N'}}(\overline{N'}_{n})}D_{(Y,\varphi,n)}(\nu,\sigma,w)\text{ where } \tilde{\chi}={^{t^n}\chi}.$$
\end{thm}
\begin{proof} We have already showed that $\tilde{\chi}$ is compatible with $w$, $w_0$ and $w_0'$. It does makes sense then to talk about $C_{\tilde{\chi}}(\nu,\sigma,w)$. 
 We also have by Lemma \ref{X,Y} and Lemma \ref{generic}, that $(\sigma,W)$ is $(Y,\varphi)$--generic and that $w$ is compatible with $(Y,\varphi)$. We assume that $n$ is sufficiently large so that $W_n$ is one dimensional which is possible by Theorem \ref{2}; and let  $v$ be a non-zero vector in $W_n$. We then construct the function $f_{(\nu,\sigma,v)}\in V(\nu,\sigma)$. Recall the definition of $f_{(\nu,\sigma,v)} $ in our situation to be the function with support in $P\overline{N}_n$ given by $f(pj)=\sigma(p)\delta_P(p)^{-1/2}\nu(p)\chi(w_{\ell}t^njt^{-n}w_{\ell}^{-1})v$, for $p\in P, j\in \overline{N}_n$. We compute $\tilde{\lambda}_{\tilde{\chi}}(\nu,\sigma)\left(I(\nu,\sigma)(w_{\ell})f_{(\nu,\sigma,v)}\right)$. We then get   
\begin{align}\tilde{\lambda}&_{\tilde{\chi}}(\nu,\sigma)\left(I(\nu,\sigma)(w_{\ell})f_{(\nu,\sigma,v)}\right)\notag\\
&=\int\limits_{w_0\overline{N}{w_0}^{-1}}\tilde{\lambda}(f(w_0^{-1} xw_{\ell}))\overline{\tilde{\chi}(x)}dx\notag \\
&=\int\limits_{w_0\overline{N}{w_0}^{-1}}\tilde{\lambda}(f({(w_{\ell}^{M})}^{-1} w_{\ell}^{-1} xw_{\ell}))\overline{\tilde{\chi}(x)}dx\notag\\
&=\delta_P({(w_{\ell}^{M})}^{-1})^{-1/2}\nu({(w_{\ell}^{M})}^{-1})\int\limits_{w_0\overline{N}{w_0}^{-1}}\tilde{\lambda}(\sigma({(w_{\ell}^{M})}^{-1})f( w_{\ell}^{-1} xw_{\ell}))\overline{\tilde{\chi}(x)}dx\notag\\
&=\delta_P({(w_{\ell})^{M}}^{-1})^{-1/2}\nu({(w_{\ell}^{M})}^{-1})\int\limits_{w_0\overline{N}_{n}w_0^{-1}}\tilde{\lambda}(\sigma({(w_{\ell}^{M})}^{-1})v)\chi(w_{\ell}t^{n}w_{\ell}^{-1}xw_{\ell}t^{-n}w_{\ell}^{-1})\overline{\tilde{\chi}(x)}dx\notag\\
&=\delta_P({(w_{\ell})^{M}}^{-1})^{-1/2}\nu({(w_{\ell}^{M})}^{-1})\int\limits_{w_0\overline{N}_{n}w_0^{-1}}\tilde{\lambda}(\sigma({(w_{\ell}^{M})}^{-1})v)\tilde{\chi}(x)\overline{\tilde{\chi}(x)}dx \text{ (by equation \ref{t})}\notag\\
&=\delta_P({(w_{\ell}^{M})}^{-1})^{-1/2}\nu({(w_{\ell}^{M})}^{-1})\mu_{w_0\overline{N}{w_0}^{-1}}(w_0\overline{N}_{n}w_0^{-1})\tilde{\lambda}(\sigma({(w_{\ell}^{M})}^{-1})v)\label{lambda f}
\end{align}
A similar computation shows 
\begin{align}&\tilde{\lambda}_{\tilde{\chi}}({^w\nu},{^w\sigma})\left(I({^w\nu},{^w\sigma})(w_{\ell})f_{({^w\nu},{^w\sigma},v)}\right)\notag\\
&=\delta_{P'}({(w_{\ell}^{M'})}^{-1})^{-1/2}{^w\nu}({(w_{\ell}^{M'})}^{-1})\mu_{w_0'\overline{N'}{w_0'}^{-1}}({w_0'}\overline{N'}_{n}{w_0'}^{-1})\tilde{\lambda}(^w\sigma({(w_{\ell}^{M'})}^{-1})v)\label{lambda of f'}
\end{align}

Since $P\cap G_n$ is a compact open subgroup of $P$ invariant under conjugation by ${(w_{\ell}^{M})}^{-1}$ we get that $\delta_{P}({(w_{\ell}^{M})}^{-1})=1$, and similarly $\delta_{P'}({(w_{\ell}^{M'})}^{-1})=1$. Using the fact that $w^{-1}w_{\ell}^{M'}w=w_{\ell}^{M}$ we get that $${^w\nu}({(w_{\ell}^{M'})}^{-1})={\nu}({(w_{\ell}^{M})}^{-1}) \text{ and }\tilde{\lambda}(^w\sigma({(w_{\ell}^{M'})}^{-1})v)=\tilde{\lambda}(\sigma({(w_{\ell}^{M})}^{-1})v)$$
We also have that the measure $\mu_{w_0\overline{N}{w_0}^{-1}}$ is that of $\mu_{\overline{N}}$ after conjugation by $w_0$. We then get that $$\mu_{w_0\overline{N}{w_0}^{-1}}(w_0\overline{N}_{n}{w_0}^{-1})=\mu_{\overline{N}}(\overline{N}_{n})$$ 
Analogously, we get that $$\mu_{w_0'\overline{N'}{w_0'}^{-1}}({w_0'}\overline{N'}_{n}{w_0'}^{-1})=\mu_{\overline{N'}}(\overline{N'}_{n})$$
Suppose for the moment that $\tilde{\lambda}(\sigma({(w_{\ell}^{M})}^{-1})v)$ is not zero. We can then divide and obtain from equation \ref{lambda f} and equation \ref{lambda of f'}  $$\frac{\tilde{\lambda}_{\tilde{\chi}}(\nu,\sigma)\left(I(\nu,\sigma)(w_{\ell})f_{(\nu,\sigma,v)}\right)}{\tilde{\lambda}_{\tilde{\chi}}({^w\nu},{^w\sigma})\left(I({^w\nu},{^w\sigma})(w_{\ell})f_{({^w\nu},{^w\sigma},v)}\right)}=\frac{\mu_{\overline{N}}(\overline{N}_{n})}{\mu_{\overline{N'}}(\overline{N'}_{n})}$$
We then get 
\begin{align*} \tilde{\lambda}_{\tilde{\chi}}(\nu,\sigma)\left(I(\nu,\sigma)(w_{\ell})f_{(\nu,\sigma,v)}\right)=&C_{\tilde{\chi}}(\nu,\sigma,w)\tilde{\lambda}_{\tilde{\chi}}({^w\nu},{^w\sigma})\left(A(\nu,\sigma,w)I(\nu,\sigma)(w_{\ell})f_{(\nu,\sigma,v)}\right)\\
=&C_{\tilde{\chi}}(\nu,\sigma,w)\tilde{\lambda}_{\tilde{\chi}}({^w\nu},{^w\sigma})\left(I({^w\nu},{^w\sigma})(w_{\ell})A(\nu,\sigma,w)f_{(\nu,\sigma,v)}\right)\\
=&C_{\tilde{\chi}}(\nu,\sigma,w)\tilde{\lambda}_{\tilde{\chi}}({^w\nu},{^w\sigma})\left(I({^w\nu},{^w\sigma})(w_{\ell})D_{(Y\varphi,n)}(\nu,\sigma,w)^{-1}f_{({w^\nu}, {w^\sigma},v)}\right)\\
=&C_{\tilde{\chi}}(\nu,\sigma,w)D_{(Y\varphi,n)}(\nu,\sigma,w)^{-1}\tilde{\lambda}_{\tilde{\chi}}({^w\nu},{^w\sigma})\left(I({^w\nu},{^w\sigma})(w_{\ell})f_{({w^\nu}, {w^\sigma},v)}\right)\\
\implies & C_{\tilde{\chi}}(\nu,\sigma,w)D_{(Y\varphi,n)}(\nu,\sigma,w)^{-1}=\frac{\mu_{\overline{N}}(\overline{N}_{n})}{\mu_{\overline{N'}}(\overline{N'}_{n})}
\end{align*}

We now show that $\tilde{\lambda}(\sigma({(w_{\ell}^{M})}^{-1})v)$ is not zero. We have by the proof of Lemma \ref{generic} that $W_{Y_\mathfrak{m},\overline{U}\cap M}$ is one dimensional. We therefore have that for $n$ sufficiently large $W_n$ is one dimensional. 
Let $W_n'=\sigma(t^n)W_n$. Consider the map $j_n':W_n'\mapsto W_{Y_\mathfrak{m},U\cap M}$ given by the inclusion of $W_n'$ into $W$ composition with the quotient map into $ W_{Y_{\mathfrak{m}},\overline{U}\cap M}$. From the proof of Theorem 1 $(ii)$ in \cite{Varma} we have that $j_n':W_n'\mapsto W_{Y_\mathfrak{m},U\cap M}$ is an isomorphism (we are assuming $n$ is sufficiently large that makes our $j'_n$ coincide with the $j'_n$ in \cite{Varma} and it is thus an isomorphism). Consider $$\lambda \in \Hom_{U\cap M}(\sigma,\chi)$$ we get $$\lambda\circ \sigma(w_{\ell}^M)^{-1}\in \Hom_{\overline{U}\cap M}(\sigma, {^{w_{\ell}^M}\chi}).$$ If we take a non zero vector $v' \in W'_n$, we get that $j'_n(v')\neq 0$ and because $\lambda\circ \sigma(w_{\ell}^M)^{-1}$ factors through $ W_{Y_{\mathfrak{m}},\overline{U}\cap M}$, we get $\lambda(\sigma(w_{\ell}^M)^{-1}v')\neq 0$. In particular $\lambda ( \sigma(w_{\ell}^M)^{-1}\sigma(t^n)v)\neq 0$. We claim that $\lambda\circ \sigma ({(w_{\ell}^M)}^{-1}t^n) \in \Hom_{\overline{U}\cap M}(\sigma, {^{w_{\ell}^M}\tilde{\chi}}) $. Let $x={(w_{\ell}^M)}^{-1}t^n$. We have that $$\lambda\circ \sigma (x) \in \Hom_{\overline{U}\cap M}(\sigma, \chi^x). $$ In order to prove the claim we will show that $\chi^{x}(\overline{u})={^{w_{\ell}^M}\tilde{\chi}(\overline{u})},\, \overline{u}\in \overline{U}\cap M$. We have 

\begin{align*}{\chi}^x(u)=&\psi(\tr( X\log({(w_{\ell}^M)}^{-1}t^{n}ut^{-n}w_{\ell}^M))= \psi(\tr({^{w_{\ell}^M}X}\log(t^{n}ut^{-n})))\\
=&\psi(\text{tr}_{\mathfrak{m}}(({^{w_{\ell}^M}X})_\mathfrak{m}\log(t^{n}ut^{-n}))=\psi(\text{tr}_{\mathfrak{m}}(Y_\mathfrak{m}\log(t^{n}ut^{-n}))\\
=&\psi(\text{tr}_{\mathfrak{m}}(\varpi^{-2n}Y_\mathfrak{m}\log(u))=\psi(\text{tr}_{\mathfrak{m}}(\varpi^{-2n}({^{w_{\ell}^M}X})_\mathfrak{m}\log(u))\\
=&\psi(\tr(\varpi^{-2n}{X}\log({(w_{\ell}^M)}^{-1}uw_{\ell}^M)))={^{w_{\ell}^M}\tilde{\chi}(\overline{u})}.
\end{align*}
This finishes the proof of the claim. We then have that since $\tilde{\lambda}\circ \sigma(w_{\ell}^M)^{-1}$ is in  $\Hom_{\overline{U}\cap M}(\sigma, {^{w_{\ell}^M}\tilde{\chi}})$, it must be a non--zero multiple of $\lambda\circ \sigma ({(w_{\ell}^M)}^{-1}t^n)$ and hence $\tilde{\lambda}(\sigma(w_{\ell}^M)^{-1}v)\neq 0$. 

\end{proof}

\begin{remark} We get that $$\frac{\mu_{\overline{N}}(\overline{N}_{n})}{\mu_{\overline{N'}}(\overline{N'}_{n})}$$ is independent of $n$. We have that the measure $\mu_N$ is the measure coming from $\mathfrak{n}$ after pulling it back from the $\log$. The measure on $\mathfrak{n}$ is the product of the measures in $F$. Once we choose a measure for $F$ we then choose a measure for $\mu_{\overline{N}}$ and the same is true for $\mu_{\overline{N'}}$. Therefore the quotient  $$\frac{\mu_{\overline{N}}(\overline{N}_{n})}{\mu_{\overline{N'}}(\overline{N'}_{n})}$$ becomes invariant. 
\end{remark}

\begin{remark} We get for $G=GL_m(F)$ that $$\frac{\mu_{\overline{N}}(\overline{N}_{n})}{\mu_{\overline{N'}}(\overline{N'}_{n})}=1.$$ We can choose the lattice $\mathcal{L}$ in this case to be $M_m(\mathfrak{o})$. We even get that $\mu_{\overline{N}}(\overline{N}_{n})=\mu_{\overline{N'}}(\overline{N'}_{n})=1$ after choosing the right measure for $F$. We have that the measure $\mu_N$ is the measure coming from $\mathfrak{n}$ after pulling it back from the $\log$. The measure on $\mathfrak{n}$ is the product of the measures in $F$. Take the measure on $F$ to be the one that is self dual with respect the character $\psi_{\varpi^{-2n}}$. We then get that the measure of the set $\mathfrak{o}$ with respect to the measure on $F$ is $q^n$. That means that the measure of $\log(N_n)=\varpi^n M_m(\mathfrak{o}) \cap \mathfrak{n}$ has measure $|\varpi|^{n \text{dim}(\mathfrak{n})}q^{n \text{dim}(\mathfrak{n})}=1$. The same argument then shows that $\mu_{\overline{N}}(\overline{N}_{n})=\mu_{\overline{N'}}(\overline{N'}_{n})=1$. One may hope $$\frac{\mu_{\overline{N}}(\overline{N}_{n})}{\mu_{\overline{N'}}(\overline{N'}_{n})}=1$$ for a general quasi-split group. 
\end{remark}

\section{The case of $\mathrm{GL}_m(D)$}
We first introduce some notation. We let $D$ be a division algebra with center $F$. We suppose that the index $[D:F]=d^2$. We have a \emph{reduced norm} $\text{rdN}:D\mapsto F^{\times}$ and a \emph{reduced trace} that we denote by $\text{rdTr}:D\mapsto F$. We can extend the valuation $v_F$ of $F$, to a discrete valuation $v_D$ of $D$ by the formula $v_D(x)=\frac{1}{d}v_F(\text{rdN}(x))$. We have an element $\varpi_D\in D$, such that $v_D(\varpi_D)=1$. The set $\mathcal{O}_D=\left\lbrace x\in D: v_D(x)\geq 0\right\rbrace$ is an $\mathcal{O}$ lattice in $D$. If we let $\widetilde{\mathcal{O}}_D=\left\lbrace x\in D: \text{rdTr}(xz)\in \mathcal{O}, \text{for all } z\in \mathcal{O}_D \right\rbrace $, we get that $\widetilde{\mathcal{O}}_D=\varpi_D^{1-d}\mathcal{O}_D$. The statements in the paragraph above can be found in \cite[Chapter 3]{Reiner}.  

Let $\text{tr}_D:D\rightarrow F$ be given by $\text{tr}_D(x)=\text{rdTr}(\varpi_D^{1-d}x )$. We then get that $\mathcal{O}_D$ is self dual with respect to $\text{tr}_D$. We know define a non-degenerate invariant bilinear form $\text{tr}_{M_m(D)}$ on $M_m(D)\times M_m(D)$. Indeed, given $A=(a_{i,j}), B=(b_{i,j})\in M_m(D)$ we let
													$$\text{tr}_{M_m(D)}(AB)=\text{tr}_D(\sum_{i=1}^{m}\sum_{k=1}^ma_{i,k}b_{ki}). $$ 
\noindent We choose the maximal split torus $A$ of $GL_m(D)$ to be those matrices in $GL_m(D)$ that are diagonal with entries in $F$. We fix the minimal parabolic subgroup $Q=LU$, to be the set of upper triangular matrices in $GL_m(D)$. Let $E_{i,j}$, be the matrix with a 1 in the $(i,j)$-entry for $1\leq i,j\leq m$, and zero everywhere else. If $Y$ is a relatively regular element in $\mathfrak{q}$, we get that $Y=\sum_{i=1}^{m-1}y_{i,i+1}E_{i,i+1}$, for some $0\neq y_{i,i+1}\in D, \,1\leq i\leq m-1$. 
Then $B_Y$ defines an alternating bilinear form on $M_m(D)/Y^{\#}\times M_m(D)/Y^{\#}$ (we refer to Section \ref{notation} for the definitions of $B_Y$ and $Y^{\#}$). We consider for a moment the element $Y=\sum_{i=1}^{m-1}E_{i,i+1}$, then for $X=(x_{i,j})\in M_m(D)$ to be in $Y^{\#}$ we must have that $YX=XY$. If we compare the entries in both sides of this last equation we see that it is necessary and sufficient  for $X$ to be upper triangular and that $x_{i,j}=x_{i+1,j+1}$, for $1\leq i,j\leq m-1$. We let $\mathfrak{c}=\left\lbrace (x_{i,j})\in M_m(D): x_{i,m}=0,\, 1\leq i\leq m \right\rbrace$, we then get that the quotient map from $M_m(D)$ onto $M_m(D)/Y^{\#}$ maps $\mathfrak{c}$ isomorphically into $M_m(D)/Y^{\#}$. In fact we have that $M_m(D)=\mathfrak{c}\oplus Y^{\#}$ and thus $B_Y$ defines a non-degenerate alternating bilinear form when restricted to $\mathfrak{c}\times \mathfrak{c}$. If $Y'$ is any relatively regular element of $M_m(D)$, we have that there is an element $t\in L$ such that $^tY=tYt^{-1}=Y'$. Therefore $Y'^{\#}=tY^{\#}t^{-1}$, and thus $^t\mathfrak{c}\oplus Y'^{\#}=M_m(D)$. We then get that $^t\mathfrak{c}$ is isomorphic under the quotient map to $M_m(D)/{Y'}^{\#}$ and hence $B_{Y'}$ defines a non-degenerate alternating form on $^t\mathfrak{c}$.       

The next theorem is just the statement that $\mathrm{GL}_m(D)$ satisfies hypothesis $\mathbf{H}_1$. 									 
\begin{thm} Let $Y$ be a relatively regular element of $M_m(D)$. We can find a lattice $\mathcal{L}$ of $M_m(D)$ that satisfies properties 1)--2) of Lemma \ref{lattice} with respect to $B_Y$, such that if $P=MN$ is a standard parabolic and if $\mathfrak{m}$ denotes the lie algebra of $M$, we get that $\mathfrak{m}\cap \mathcal{L}$ satisfies properties 1)--2) of Lemma \ref{lattice} with respect to $B_{Y_\mathfrak{m}}$. Moreover, we can choose $\mathcal{L}$ to also satisfy $^w\mathcal{L}=\mathcal{L}$, for some full set of representatives of $\widetilde{w}\in {_GW}$.
\end{thm}

\begin{proof} We consider first the case where $Y=\sum_{i=1}^{m-1}E_{i,i+1}$. We claim that in this case $M_m(\mathcal{O}_D)$ satisfies the required conditions.  We see with almost no effort that $M_m(\mathcal{O}_D)$ satisfies condition 1) of Lemma \ref{lattice}. To prove that $M_m(\mathcal{O}_D)$ satisfies condition $2)$ of Lemma \ref{lattice} we make use of the notation above. We need to show that $M_m(\mathcal{O}_D)\cap \mathfrak{c}$ is self dual with respect to $B_Y$. Let $B=(b_{i,j})\in \mathfrak{c}$ be such that for all $A=(a_{i,j})\in M_m(\mathcal{O}_D)\cap \mathfrak{c}$, $B_Y(A,B)\in \mathcal{O}$. Recall that the last column of the elements in $\mathfrak{c}$ is zero. We then have 

\begin{align} 
B_Y(A,B)&=\text{tr}_{M_m(D)}(Y[A,B])=\text{tr}_D(\sum_{i=2}^{m}\sum_{k=1}^{m}a_{i,k}b_{k,i-1}-b_{i,k}a_{k,i-1})\notag\\
&=\text{tr}_D(\sum_{i=2}^{m}\sum_{k=1}^{m-1}a_{i,k}b_{k,i-1}-\sum_{i=2}^{m}\sum_{k=1}^{m-1}b_{i,k}a_{k,i-1})\, (\text{last column is zero)}\notag\\
&=\text{tr}_D(\sum_{i=2}^{m}\sum_{k=1}^{m-1}a_{i,k}b_{k,i-1}-\sum_{k=2}^{m}\sum_{i=1}^{m-1}b_{k,i}a_{i,k-1})\,(k\leftrightarrow i, \text{in the second sum})\notag\\
&=\text{tr}_D(\sum_{i=2}^{m}\sum_{k=1}^{m-1}a_{i,k}b_{k,i-1}-\sum_{k=1}^{m-1}
\sum_{i=1}^{m-1}b_{k+1,i}a_{i,k})\notag\\   
&=\sum_{i=2}^{m}\sum_{k=1}^{m-1}\text{tr}_D(a_{i,k}b_{k,i-1})-\sum_{i=1}^{m-1}\sum_{k=1}^{m-1}\text{tr}_D(a_{i,k}b_{k+1,i})\notag\\
&=\sum_{i=2}^{m-1}\sum_{k=1}^{m-1}\text{tr}_D(a_{i,k}(b_{k,i-1}-b_{k+1,i}))+\sum_{k=1}^{m-1}\text{tr}_D(a_{m,k}b_{k,m-1})-\sum_{k=1}^{m-1}\text{tr}_D(a_{1,k}b_{k+1,1})\label{equation} \tag{**}
\end{align}   
We see from this last equation and from the fact that $\mathcal{O}_D$ is self dual with respect to $\text{tr}_D$, that $b_{k,m-1},b_{k+1,1}\in \mathcal{O}_D$ and that $b_{k,i-1}-b_{k+1,i}\in \mathcal{O}_D$, for $1\leq k\leq m-1,\,\, 2\leq i\leq m-1$. We conclude from these recursive relations that $b_{k,i}\in \mathcal{O}_D$ for $1\leq i,k\leq m$. If $P=MN$ is any standard parabolic subgroup with lie algebra $\mathfrak{m}$, we have that $$\mathfrak{m}\cong M_{m_1}(D)\oplus M_{m_2}(D)\oplus\cdots \oplus M_{m_k}(D)$$ for some positive integers $m_1,m_2,\ldots m_k$. We then get that $$\mathfrak{m}\cap M_m(\mathcal{O_D})\cong M_{m_1}(\mathcal{O}_D)\oplus M_{m_2}(\mathcal{O}_D)\oplus\cdots \oplus M_{m_k}(\mathcal{O}_D).$$ It then follows from what we just proved that $\mathfrak{m}\cap M_m(\mathcal{O_D})$ satisfies properties 1) and 2) of Lemma \ref{lattice}. The general case now follows by the fact that every other relatively regular element is obtained by conjugating $Y$ by an element of the minimal Levi, and thus conditions 1) and 2) of Lemma \ref{lattice} are satisfied by transport de structure. 
\end{proof}

\begin{thm} The group $G=\mathrm{GL}_m(D)$ satisfies hypothesis $\mathbf{H}_2$. 
\end{thm}

\begin{proof} Let $(\sigma,W) $ be a $(Y,\varphi)$--generic representation. Let $\mathcal{O}_{Y_{\mathfrak{m}}}$ be the $M$--orbit of $Y_{\mathfrak{m}}$ under the adjoint action. Let $C_{\mathcal{O}_{Y_{\mathfrak{m}}}}$ be the coefficient in the Harish-Chandra character expansion of $(\sigma\otimes \nu ,W)$ ($C_{\mathcal{O}_{Y_{\mathfrak{m}}}}$ is independent of $\nu$). Likewise, let $C_{\mathcal{O}_{Y}}$ be the coefficient in the Harish-Chandra character expansion of $I(\nu,\sigma)$. Let $\mathcal{O}(\mathfrak{m})$ denote the set of nilpotent orbits in $\mathfrak{m}$, and let $\mathcal{O}(\mathfrak{g})$ denote the nilpotent orbits in $\mathfrak{g}$. For an orbit $\mathcal{O}_{\mathfrak{m}}\in \mathcal{O}(\mathfrak{m}) $ we define the set $\text{ind}_M^G(\mathcal{O}_{\mathfrak{m}})$ to be the set of all orbits $\mathcal{O}\in \mathcal{O}(\mathfrak{g})$ such that $\mathcal{O}\cap \left(\mathcal{O}_\mathfrak{m}+\mathfrak{n}\right)$
is open set of $\mathcal{O}_\mathfrak{m}+\mathfrak{n}$. Equation (9) in \cite{MW} (please see remark) with our notation reads 
$$\text{tr}(I(\nu,\sigma)(f))=\sum\limits_{\mathcal{O}_{\mathfrak{m}}\in \mathcal{O}(\mathfrak{m})}\sum\limits_{\mathcal{O}\in \text{ind}_M^G(\mathcal{O}_{\mathfrak{m}})}C_{\mathcal{O}_{\mathfrak{m}}}|Stab_{G}(X_{\mathcal{O}})/Stab_P(X_{\mathcal{O}})|\widehat{\mu}_{\mathcal{O}}(f\circ \exp)$$  
Where $X_{\mathcal{O}}$ is any element in $\mathcal{O}\cap \left(\mathcal{O}_\mathfrak{m}+\mathfrak{n}\right)$, and $Stab_{G}(X_{\mathcal{O}})$ ( resp. $Stab_P(X_{\mathcal{O}})$) is the stabilizer of $X_{\mathcal{O}}$ under the adjoin action of $G$ (resp.  the adjoint action of $P$).  Let $\mathcal{E}$ be the set of $\mathcal{O}_{\mathfrak{m}}\in \mathcal{O}(\mathfrak{m})$ such that $\mathcal{O}_{Y}\in \text{ind}_M^G(\mathcal{O}_{\mathfrak{m}})$.  We therefore obtain that $$C_{\mathcal{O}_{Y}}=\sum\limits_{\mathcal{E}}C_{\mathcal{O}_{\mathfrak{m}}}|Stab_{G}(X_{\mathcal{O}})/Stab_P(X_{\mathcal{O}})|.$$

Consider $\mathcal{O}_{\mathfrak{m}}\in \mathcal{E}$, then $\mathcal{O}_{Y}\cap \left(\mathcal{O}_\mathfrak{m}+\mathfrak{n}\right)$ is open and therefore from (7) in \cite[pg. 443]{MW} is a $P$ orbit.
 From Lemma 1 in \cite{Rao} we get that $P_{0}$ orbit of $Y$ which is contained in the $P$ orbit of $Y$ contains an element $X$ that is upper triangular and whose (i,i+1)-entries are non-zero for $1\leq i \leq m-1$. We have that $X\in \mathcal{O}_{Y}\cap \left(\mathcal{O}_\mathfrak{m}+\mathfrak{n}\right) $, so if we write $X=X_{\mathfrak{m}}+N$, for $X_{\mathfrak{m}}\in \mathcal{O}_{\mathfrak{m}}$ and $N \in \mathfrak{n}$, we get that the (i,i+1)-entries of $X_{\mathfrak{m}}$ are not-zero and $X_{\mathfrak{m}}$ is upper triangular.  Using again Lemma 1 in \cite{Rao}  we get that $X_{\mathfrak{m}}$ is in the orbit of $Y_{\mathfrak{m}}$ and therefore $\mathcal{O}_{\mathfrak{m}}=\mathcal{O}_{Y_\mathfrak{m}}$. That means that $\mathcal{E}$ consists of the single orbit $\mathcal{O}_{Y_\mathfrak{m}}$. Hence
 $$C_{\mathcal{O}_Y}=C_{\mathcal{O}_{Y_\mathfrak{m}}}|Stab_{G}(Y)/Stab_P(Y)|.$$ 
We claim that $Stab_{G}(Y)\subset P$, which implies $C_{\mathcal{O}_{Y_\mathfrak{m}}}=C_{\mathcal{O}_Y}$. To prove the claim we first suppose that $Y$ is the relatively regular element with $1's$ in the entries above the diagonal zero everywhere else. Let $g\in G$ be such that $gY=Yg$. After comparing entries of the last equation we arrive to the fact that $g$ is upper triangular and thus contained in the minimal parabolic $Q=LU$. Since any other relatively regular element is a conjugate of $Y$ by an element in $L$, we get that $Stab_G(Y)\subset Q\subset P$. This finishes the proof of the theorem. 
\end{proof}

\begin{remark} In the proof of equation (9) in \cite{MW} they use the result \cite[IV.2.25]{Springer&Steinberg} and from there they conclude that for a nilpotent element $X$ in the lie algebra of $G$, the natural map $Stab_{G}(X)$ into $Stab_{\mathbb{G}(\overline{F})}(X)/Stab_{\mathbb{G}(\overline{F})}(X)^{\circ}$ is surjective. Here $\mathbb{G}(\overline{F})$ denotes the points of the algebraically close field $\overline{F}$ and $Stab_{\mathbb{G}(\overline{F})}(X)^{\circ}$ is the connected component of $Stab_{\mathbb{G}(\overline{F})}(X)$. The result \cite[IV.2.25]{Springer&Steinberg} is stated for classical groups and we can not use this result for the case $G=\text{GL}_m(D)$. However in the case $G=\text{GL}_m(D)$ we get that $Stab_{\mathbb{G}(\overline{F})}(Y)$ is connected and therefore the surjectivity of  $Stab_{G}(X)$ into $Stab_{\mathbb{G}(\overline{F})}(X)/Stab_{\mathbb{G}(\overline{F})}(X)^{\circ}$ is trivial. To see that that $Stab_{\mathbb{G}(\overline{F})}(X)$ is connected in the Zariski topology we first see that $\mathbb{G}(\overline{F})=\text{GL}_{md}(\overline{F})$. Let $M_{md}(\overline{F})$ be the set of $md\times md$ matrices over $\overline{F}$ and consider $X$ as an element in $M_{md}(\overline{F})$. The solutions so the equation $XZ=ZX$ for $Z\in M_{md}(\overline{F})$ is an affine space. We get that the set of matrices that satisfies $XZ=ZX$ and $\det(Z)\neq 0$ is a principal open set in affine space and thus connected.   
\end{remark}

\bibliographystyle{amsalpha}
\bibliography{Coefficients}{}

\end{document}